\documentclass[a4paper]{amsart}

\usepackage{mathrsfs}
\usepackage{amsfonts}
\usepackage{amssymb}
\usepackage{amsxtra}
\usepackage{dsfont}
\usepackage{graphicx}
\usepackage{lmodern}

\usepackage[colorlinks,citecolor=blue,urlcolor=blue,bookmarks=true]{hyperref}
\hypersetup{
pdfpagemode=UseNone,
pdfstartview=FitH,
pdfdisplaydoctitle=true,
pdfborder={0 0 0}, 
pdftitle={Littlewood--Paley theory for triangle buildings},
pdfauthor={Tim Steger and Bartosz Trojan},
pdflang=en-GB
}

\usepackage[english,polish]{babel}

\newcommand{\defn}[1]{\emph{#1}}

\newcommand{\sprod}[2]{{\langle #1, #2\rangle}}
\newcommand{\norm}[1]{{\left\lVert #1 \right\rVert}}
\newcommand{\abs}[1]{{\lvert {#1} \rvert}}

\newcommand{\pnorm}[2]{{\norm{{#1}}_{L^{#2} \rightarrow L^{#2}}}}

\newcommand{\ind}[1]{{\mathbf{1}_{{#1}}}}

\newcommand{\lell}[2]{L^{#1}\left(\pi, \ell^{#2}(P) \right)}

\newcommand{\GL}{\operatorname{GL}}
\newcommand{\Aff}{\operatorname{Aff}}

\newcommand{\seq}[2]{\left({#1}: {#2}\right)}

\newcommand{\dpi}{{\: \rm d}\pi}

\newcommand{\EE}{\mathbb{E}}

\newcommand{\RR}{\mathbb{R}}
\newcommand{\QQ}{\mathbb{Q}}
\newcommand{\ZZ}{\mathbb{Z}}
\newcommand{\NN}{\mathbb{N}}
\newcommand{\calF}{\mathcal{F}}
\newcommand{\calS}{\mathcal{S}}
\newcommand{\calO}{\mathcal{O}}

\newcommand{\pl}[1]{\foreignlanguage{polish}{#1}}
\renewcommand{\atop}[2]{\substack{{#1}\\{#2}}}

\theoremstyle{plain}
\newtheorem{theorem}{Theorem}
\newtheorem{proposition}{Proposition}[section]

\newtheorem{lemma}[proposition]{Lemma}
\theoremstyle{definition}

\theoremstyle{remark}

\numberwithin{equation}{section}

\newcounter{thm}

\theoremstyle{plain}
\newtheorem{main_theorem}[thm]{Theorem}

\usepackage[margin=3cm, centering]{geometry}

\begin{document}
\selectlanguage{english}

\title[Littlewood--Paley theory for buildings]
	{Littlewood--Paley theory for triangle buildings}

\author{Tim Steger}
\address{Tim Steger\\
		 Matematica\\
		 Universit\`a degli Studi di Sassari\\
		 Via Piandanna 4\\
		 07100 Sassari\\
		 Italy}
\email{steger@uniss.it}

\author{Bartosz Trojan}
\address{
        \pl{
        Bartosz Trojan\\
        Wydzia\l{} Matematyki,
        Politechnika Wroc\l{}awska\\
        Wyb. Wyspia\'{n}skiego 27\\
        50-370 Wroc\l{}aw\\
        Poland}
}
\email{bartosz.trojan@pwr.edu.pl}

\begin{abstract}
For the natural two parameter filtration
$\seq{\mathcal{F}_\lambda}{\lambda \in P}$ on the boundary of a
triangle building we define a maximal function and a square function
and show their boundedness on $L^p(\Omega_0)$ for $p \in (1, \infty)$.
At the end we consider $L^p(\Omega_0)$ boundedness of martingale
transforms. If the building is of $\GL(3, \QQ_p)$ then $\Omega_0$ can be
identified with $p$-adic Heisenberg group.
\end{abstract}



\maketitle

\section{Introduction}
Let $(\Omega, \calF, \pi)$ be a $\sigma$-finite measure space. A sequence of $\sigma$-algebras $(\calF_n : n \in \ZZ)$
is a filtration if $\calF_n \subset \calF_{n+1}$. Given $f$ a locally integrable function on $\Omega$ by
$\EE[f | \calF_n]$ we denote its conditional expectation value with respect to $\calF_n$. Let $M^*$ and $S$ denote
the maximal function and the square function defined by
\[
	M^* f = \sup_{n \in \ZZ} \abs{f_n},
\]
and
\begin{equation}
	\label{eq:32}
	S f = \Big(\sum_{n \in \ZZ} \abs{d_n f}^2 \Big)^{1/2},
\end{equation}
where $d_n f = f_n - f_{n-1}$. The Hardy and Littlewood maximal estimate (see \cite{hl}) implies that
\[
	\pi\big(\big\{ M^* f > \lambda \big\}\big) \leq \lambda^{-1} \int_{M^* f > \lambda} \abs{f} \dpi,
\]
from where it is easy to deduce that for $p \in (1, \infty]$
\[
	\norm{M^* f}_{L^p} \leq \frac{p}{p-1} \norm{f}_{L^p}.
\]
For the square function, if $p \in (1, \infty)$ then there is $C_p > 1$
such that
\begin{equation}
	\label{eq:31}
	C_p^{-1} \norm{f}_{L^p} \leq \norm{S f}_{L^p} \leq C_p \norm{f}_{L^p}.
\end{equation}
The inequality \eqref{eq:31} goes back to Paley \cite{pal}, and has been reproved in many ways, see for example
\cite{burk1, burk2, burk3, gundy2, mar}. Its main application is in proving the
$L^p$-boundedness of martingale transforms (see \cite{burk1}), that is, for operators of the form
\[
	T f = \sum_{n \in \ZZ} a_n d_n f
\]
where $(a_n : n \in \ZZ)$ is a sequence of uniformly bounded functions such that $a_{n+1}$ is $\calF_n$-measurable. 

In 1975, Cairoli and Walsh (see \cite{carwal}) have started to generalize the theory of martingales to two parameter case.
Let us recall that a sequence of $\sigma$-fields $(\calF_{n, m} : n, m \in \ZZ)$ is a two parameter filtration if
\begin{equation}
	\label{eq:34}
	\calF_{n+1, m} \subset \calF_{n, m}, \qquad \text{and} \qquad
	\calF_{n, m+1} \subset \calF_{n, m}.
\end{equation}
Then $(f_{n, m} : n, m \in \ZZ)$ is a two parameter martingale if
\begin{equation}
	\label{eq:33}
	\EE[f_{n+1, m} | \calF_{n, m}] = f_{n, m}, \qquad \text{and} \qquad
	\EE[f_{n, m+1} | \calF_{n, m}] = f_{n, m}.
\end{equation}
Observe that conditions \eqref{eq:34} and \eqref{eq:33} impose a structure only for comparable indices. 
In that generality, it is hard, if not impossible, to build the Littlewood--Paley theory. This lead to the introduction
of other (smaller) classes of martingales (see \cite{zak, wongzak}). In particular, in \cite{carwal}, Cairoli and Walsh
introduced the following condition
\begin{equation}
	\label{f4}
	\tag{$F_4$}
	\EE[f | \calF_{n, \infty} | \calF_{\infty, m} ] = 
	\EE[f | \calF_{\infty,m} | \calF_{n, \infty} ] = f_{n, m}
\end{equation}
where
\[
	\calF_{n, \infty} = \sigma\Big(\bigcup_{m \in \ZZ} \calF_{n ,m} \Big),
	\qquad \text{and} \qquad
	\calF_{\infty, m} = \sigma\Big(\bigcup_{n \in \ZZ} \calF_{n, m} \Big).
\]
Under \eqref{f4}, the result obtained by Jensen, Marcinkiewicz and Zygmund in \cite{jmz} implies that
the maximal function
\begin{equation}
	\label{eq:35}
	M^* f = \sup_{n, m \in \ZZ} \abs{f_{n, m}}
\end{equation}
is bounded on $L^p(\Omega)$ for $p \in (1, \infty]$. In this context the square function is defined by
\begin{equation}
	\label{eq:36}
	S f = \Big(\sum_{n, m \in \ZZ} \abs{d_{n, m} f}^2 \Big)^{1/2}
\end{equation}
where $d_{n, m}$ denote the double difference operator, i.e.
\[
	d_{n, m} f = f_{n, m} - f_{n-1, m} - f_{n, m-1} + f_{n-1, m-1}.
\]
In \cite{mat}, it was observed by Metraux that the boundedness of $S$ on $L^p(\Omega)$ for $p \in (1, \infty)$
is implied by the one parameter Littlewood--Paley theory. Also the concept of a martingale transform has a
natural generalization, that is,
\[
	T f = \sum_{n, m \in \ZZ} a_{n, m} d_{n, m} f
\]
where $(a_{n, m} : n,m \in \ZZ)$ is a sequence of uniformly bounded functions such that $a_{n+1, m+1}$ is 
$\calF_{n, m}$-measurable.

In this article we are interested in a case when the condition \eqref{f4} is not satisfied. The simplest example 
may be obtained by considering the Heisenberg group together with the non-isotropic two parameter dilations
\[
	\delta_{s, t} (x, y, z) = (s x, t y, st z).
\]
Since in this setup the dyadic cubes do not posses the same properties as the Euclidean cubes, it is more convenient
to work on the $p$-adic version of the Heisenberg group. We observe that this group can be identified with $\Omega_0$,
a subset of a boundary of the building of $\GL(3, \QQ_p)$ consisting of the points opposite to a given $\omega_0$. The set
$\Omega_0$ has a natural two parameter filtration $(\calF_{n, m} : n, m \in \ZZ)$ (see Section \ref{sec:2} for
details). The maximal function and the square function are defined by \eqref{eq:35} and \eqref{eq:36}, respectively.
The results we obtain are summarized in the following three theorems.
\begin{main_theorem}
	\label{thm:1}
	For each $p \in (1, \infty]$ there is $C_p > 0$ such that for all $f \in L^p\big(\Omega_0\big)$
	\[
		\norm{M^*f}_{L^p} \leq C_p \norm{f}_{L^p}.
	\]
\end{main_theorem}
\begin{main_theorem}
	\label{thm:2}
	For each $p \in (1, \infty)$ there is $C_p > 1$ such that for all $f \in L^p\big(\Omega_0\big)$
	\[
		C_p^{-1} \norm{f}_{L^p} \leq \norm{S f}_{L^p} \leq C_p \norm{f}_{L^p}.
	\]
\end{main_theorem}
\begin{main_theorem}
	\label{thm:3}
	If $(a_{n, m} : n, m \in \ZZ)$ is a sequence of uniformly bounded functions such that $a_{n+1, m+1}$ is
	$\calF_{n, m}$-measurable, then the martingale transform
	\[
		T f = \sum_{n, m \in \ZZ} a_{n, m} d_{n, m} f 
	\]
	is bounded on $L^p\big(\Omega_0\big)$, for all $p \in (1, \infty)$.
\end{main_theorem}
Let us briefly describe methods we use. First, we observe that instead of \eqref{f4} the stochastic basis
satisfies the remarkable identity \eqref{eq:23}. Based on it we show that the following pointwise estimate holds
\begin{equation}
	\label{eq:37}
	M^*(\abs{f}) \leq C \big(L^* R^* L^* R^*(\abs{f}) + R^* L^* R^* L^* (\abs{f}) \big)
\end{equation}
proving the maximal theorem. Thanks to the two parameter Khintchine's inequality, to bound 
the square function $S$, it is enough to show Theorem \ref{thm:3}. To do so, we define a new square function $\calS$
which has a nature similar to the square function used in the presence of \eqref{f4}. Then we adapt the technique
developed by Duoandikoetxea and Rubio de Francia in \cite{DuoRdF} (see Theorem \ref{th:6}). This implies
$L^p$-boundedness of $S$. Since $S$ does not preserve the $L^2$ norm, the lower bound requires an extra argument.
Namely, we view the square function $S$ as an operator with values in $L^p(\ell^2)$ and take its dual. As a consequence
of Theorem \ref{th:6} and the identity \eqref{eq:48} the later is bounded on $L^p$. 

Finally, let us comment on the behavior of the maximal function $M^*$ close to $L^1$. Based on the pointwise estimate
\eqref{eq:37}, in view of \cite{hl}, we conclude that $M^*$ is of weak-type for functions in the Orlicz space
$L (\log L)^3$. To better understand the maximal function $M^*$ we investigate exact behavior close to $L^1$. This
together with weighted estimates is the subject of the forthcoming paper. It is also interesting how to extend
theorems \ref{thm:1}, \ref{thm:2} and \ref{thm:3} to higher rank and other types of affine buildings.
 
\subsection{Notation}
For two quantities $A>0$ and $B>0$ we say that $A \lesssim B$ ($A \gtrsim B$) if there exists 
an absolute constant $C>0$ such that $A\le CB$ ($A\ge CB$).

If $\lambda \in P$ we set $\abs{\lambda} = \max\{ \abs{\lambda_1}, \abs{\lambda_2}\}$.

\section{Triangle buildings}
\label{sec:2}
\subsection{Coxeter complex}
We recall basic facts about the $A_2$ root system and the
$\tilde{A}_2$ Coxeter group.  A general reference is \cite{bour}.  Let
$\mathfrak{a}$ be the hyperplane in $\RR^3$ defined as
\begin{equation*}
\mathfrak{a} = \{(x_1, x_2, x_3) \in \RR^3: x_1 + x_2 + x_3 = 0 \}.
\end{equation*}
We denote by $\{e_1, e_2, e_3\}$ the canonical orthonormal basis of
$\RR^3$ with respect to the standard scalar product
$\sprod{\cdot}{\cdot}$. We set $\alpha_1 = e_2 - e_1$, $\alpha_2 = e_3
- e_2$, $\alpha_0 = e_3 - e_1$ and $I = \{0, 1, 2 \}$. The $A_2$ root
system is defined by
\begin{equation*}
	\Phi = \{\pm\alpha_0, \pm \alpha_1, \pm \alpha_2\}.
\end{equation*}
We choose the base $\{\alpha_1, \alpha_2\}$ of~$\Phi$.  The
corresponding positive roots are $\Phi^+ = \{\alpha_0, \alpha_1,
\alpha_2\}$. Denote by $\{\lambda_1, \lambda_2\}$ the basis dual to
$\{\alpha_1, \alpha_2\}$; its elements are called the
\defn{fundamental co-weights}. Their integer combinations, form the
\defn{co-weight lattice}~$P$.
\begin{figure}[h]
	\includegraphics{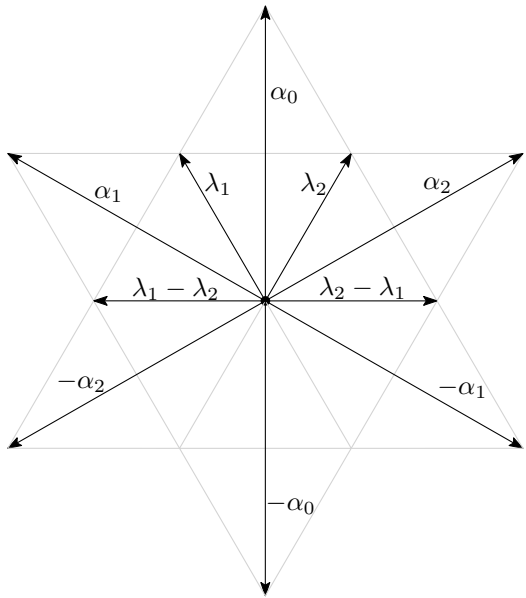}
	\caption{$A_2$ root system}
        \label{fig:1}
\end{figure}
As in \textsc{Figure \ref{fig:1}}, we always draw $\lambda_1$ pointing
up and to the left and $\lambda_2$ up and to the right.  Likewise
$\lambda_1-\lambda_2$ is drawn pointing directly left, while
$\lambda_2-\lambda_1$ points directly right.  Because
$\sprod{\lambda_1}{\alpha_0}=\sprod{\lambda_2}{\alpha_0}=1$,
we see that for any~$\lambda\in P$ the expression
$\sprod{\lambda}{\alpha_0}$ represents the vertical level
of~$\lambda$.  For~$\lambda=i\lambda_1+j\lambda_2$, that level
is~$i+j$.

Let $\mathcal{H}$ be the family of affine hyperplanes, called
\defn{walls},
\begin{equation*}
H_{j; k}=\{x \in \mathfrak{a}: \sprod{x}{\alpha_j} = k\}
\end{equation*}
where $j \in I$, $k \in \ZZ$. To each wall $H_{j; k}$ we associate
$r_{j;k}$ the orthogonal reflection in $\mathfrak{a}$, i.e.
\begin{equation*}
r_{j; k} (x) = x - \big(\sprod{x}{\alpha_j} - k\big) \alpha_j.
\end{equation*}
Set $r_1 = r_{1; 0}$, $r_2 = r_{2; 0}$ and $r_0 = r_{0; 1}$. The
\emph{finite Weyl group} $W_0$ is the subgroup of $\GL(\mathfrak{a})$
generated by $r_1$ and $r_2$. The \emph{affine Weyl group} $W$ is the
subgroup of $\Aff(\mathfrak{a})$ generated by $r_0$, $r_1$ and $r_2$.

Let $\mathcal{C}$ be the family of open connected components of
$\mathfrak{a} \setminus \bigcup_{H \in \mathcal{H}} H$. The elements
of $\mathcal{C}$ are called \emph{chambers}. By $C_0$ we denote the
fundamental chamber, i.e.
\begin{equation*}
	C_0 = \{x \in \mathfrak{a}: \sprod{x}{\alpha_1} > 0,
        \sprod{x}{\alpha_2} > 0, \sprod{x}{\alpha_0} < 1\}.
\end{equation*}
The group $W$ acts simply transitively on $\mathcal{C}$. Moreover,
$\overline{C_0}$ is a fundamental domain for the action of $W$ on
$\mathfrak{a}$ (see e.g. \cite[VI, \S 1-3]{bour}).  The vertices of
$C_0$ are $\{0, \lambda_1, \lambda_2\}$.  The set of all vertices of
all~$C\in\mathcal{C}$ is denoted by~$V(\Sigma)$.  Under the action
of~$W$, $V(\Sigma)$ is made up of three orbits, $W(0)$,
$W(\lambda_1)$, and~$W(\lambda_2)$.  Vertices in the same orbit are
said to have the same~\defn{type}.  Any chamber $C\in\mathcal{C}$ has
one vertex in each orbit or in other words one vertex of each of the
three types.

The family $\mathcal{C}$ may be regarded as a simplicial complex
$\Sigma$ by taking as the simplexes all non-empty subsets of vertices
of~$C$, for all $C \in \mathcal{C}$.  Two chambers $C$ and $C'$ are
\defn{$i$-adjacent} for $i \in I$ if $C = C'$ or if there is $w \in W$
such that $C=wC_0$ and $C'=wr_iC_0$.  Since $r_i^2=1$ this defines an
equivalence relation.

The \defn{fundamental sector} is defined by 
\begin{equation*}
	\mathcal{S}_0 = \{x \in \mathfrak{a}: \sprod{x}{\alpha_1} > 0, \sprod{x}{\alpha_2} > 0\}.
\end{equation*}
Given $\lambda \in P$ and $w \in W_0$ the set $\lambda + w
\mathcal{S}_0$ is called a \defn{sector} in $\Sigma$ with \defn{base
  vertex} $\lambda$.  The angle spanned by a sector at its base vertex
is~$\pi/3$.

\subsection{The definition of triangle buildings}
For the theory of affine buildings we refer the reader to \cite{ron}. 
See also the first author's expository paper \cite{st-b}, for an elementary introduction to the $p$-adics,
and to precisely the sort of the buildings which this paper deals with.

A simplicial complex
$\mathscr{X}$ is an \defn{$\tilde A_2$~building}, or as we like to call it, a
\defn{triangle building}, if each of its vertices is assigned one of the three types, and if
it contains a family of subcomplexes called {\em apartments} such that
\begin{enumerate}
  \item each apartment is type-isomorphic to~$\Sigma$,
  \item any two simplexes of $\mathscr{X}$ lie in a common apartment,
  \item\label{axiom3} for any two apartments $\mathscr{A}$ and
    $\mathscr{A}'$ having a chamber in common there is a
    type-preserving isomorphism $\psi: \mathscr{A} \rightarrow
    \mathscr{A}'$ fixing $\mathscr{A} \cap \mathscr{A}'$ pointwise.
\end{enumerate}
We assume also that the system of apartments is \defn{complete},
meaning that any subcomplex of~$\mathscr{X}$ type-isomorphic
to~$\Sigma$ is an apartment.  A simplex~$C$ is a \defn{chamber}
in~$\mathscr{X}$ if it is a chamber for some apartment. Two chambers
of~$\mathscr{X}$ are \defn{$i$-adjacent} if they are $i$-adjacent in
some apartment. For $i\in I$ and for a chamber~$C$ of~$\mathscr{X}$
let $q_i(C)$ be equal to
\begin{equation*}
	q_i(C) = \abs{\{C' \in \mathscr{X}: C' \sim_i C\}} - 1.
\end{equation*}
It may be proved that $q_i(C)$ is independent of~$C$ and
of~$i$. Denote the common value by~$q$, and assume local finiteness:
$q<\infty$.  Any \emph{edge} of $\mathscr{X}$, i.e., any $1$-simplex,
is contained in precisely~$q+1$ chambers.

It follows from the axioms that the ball of radius one about any
vertex~$x$ of~$\mathscr{X}$ is made up of~$x$ itself, which is of one
type, $q^2+q+1$~vertices of a second type, and a further~$q^2+q+1$
vertices of the third type.  Moreover, adjacency between vertices of
the second and third types makes them into, respectively, the points
and the lines of a finite projective plane.

A subcomplex $\mathscr{S}$
is called a \defn{sector} of~$\mathscr{X}$ if it is a sector in some
apartment.  Two sectors are called \defn{equivalent} if they contain a
common subsector.  Let~$\Omega$ denote the set of equivalence classes
of sectors.  If~$x$ is a vertex of~$\mathscr{X}$ and
$\omega\in\Omega$, there is a unique sector denoted $[x,\omega]$ which
has base vertex~$x$ and represents~$\omega$.

Given any two points $\omega$ and~$\omega' \in \Omega$, one can find
two sectors representing them which lie in a common apartment. If that
apartment is unique, we say that~$\omega$ and~$\omega'$ are
\defn{opposite}, and denote the unique apartment by
$[\omega,\omega']$.  In fact $\omega$ and~$\omega'$ are opposite
precisely when the two sectors in the common apartment point in
opposite directions in the Euclidean sense.

\subsection{Filtrations}
We fix once and for all an origin vertex~$O \in \mathscr{X}$ and a
point~$\omega_0 \in \Omega$.  Choose~$O$ so that it has the same type
as the origin of~$\Sigma$. Let $\mathscr{S}_0=[O,\omega_0]$ be the
sector representing~$\omega_0$ with base vertex~$O$. By~$\Omega_0$ we
denote the subset of $\Omega$ consisting of $\omega$'s opposite to
$\omega_0$.  For purposes of motivation only, we recall that if
$\mathscr{X}$ is the building of $\GL(3,\QQ_p)$, then $\Omega_0$ can
be identified with the $p$-adic Heisenberg group (see Appendix \ref{ap:1} for details).

Let $\mathscr{A}_0$ be any apartment containing $\mathscr{S}_0$.  By
$\psi$ we denote the type-preserving isomorphism between $\mathscr{A}_0$ and $\Sigma$
such that $\psi(\mathscr{S}_0) = -S_0$. We set $\rho = \psi \circ
\rho_0$ where $\rho_0$ is the retraction from $\mathscr{X}$ to
$\mathscr{A}_0$.  With these definitions, $\rho:\mathscr{X}\to\Sigma$
is a type-preserving simplicial map, and for any~$\omega\in\Omega_0$
the apartment~$[\omega,\omega_0]$ maps bijectively to~$\Sigma$
with~$\omega_0$ mapping to the bottom (of \textsc{Figure~\ref{fig:1}})
and $\omega$ mapping to the top.

For any vertex~$x$ of~$\mathscr{X}$ define the subset $E_x \subset
\Omega_0$ to consist of all~$\omega$'s such that $x$~belongs to
$[\omega, \omega_0]$; an equivalent condition is that
$[x,\omega_0]\subseteq[\omega,\omega_0]$.  Fix $\lambda \in P$.  By
$\mathcal{F}_\lambda$ we denote the $\sigma$-field generated by sets
$E_x$ for $x \in \mathscr{X}$ with $\rho(x) = \lambda$. There are
countably many such~$x$, and the corresponding sets~$E_x$ are mutually
disjoint, hence $\mathcal{F}_\lambda$ is a countably generated atomic
$\sigma$-field.

Let $\preceq$ denote the partial order on $P$ where $\lambda \preceq
\mu$ if and only if $\sprod{\lambda - \mu}{\alpha_1} \leq 0$ and
$\sprod{\lambda - \mu}{\alpha_2} \leq 0$.  If we draw and
orient~$\Sigma$ as in \textsc{Figure \ref{fig:1}}, then $\lambda
\preceq \mu$ exactly when $\mu$~lies in the sector pointing
upwards from~$\lambda$.

\begin{proposition}
\label{prop:0}
If $\lambda \preceq \mu$ then $\mathcal{F}_\lambda \subset
\mathcal{F}_\mu$.
\end{proposition}
\begin{proof}
Choose any vertex~$x$ so that $\rho(x)=\mu$. Because $\lambda \preceq
\mu$, there is a unique vertex~$y$ in the sector $[x,\omega_0]$ so
that $\rho(y)=\lambda$.  For any $\omega\in E_x$, the apartment
$[\omega,\omega_0]$ contains~$x$, hence it contains $[x,\omega_0]$,
hence it contains~$y$.  This establishes that $E_x\subseteq E_y$.  In
other words, each atom of~$\mathcal{F}_\mu$ is a subset of some atom
of~$\mathcal{F}_\lambda$.  Hence each atom of~$\mathcal{F}_\lambda$ is
a disjoint union of atoms of~$\mathcal{F}_\mu$.
\end{proof}
In fact, Proposition \ref{prop:0} says that $\seq{\mathcal{F}_\lambda}{\lambda\in P}
=\seq{\mathcal{F}_{i\lambda_1+j\lambda_2}}{i,j\in\ZZ}$ is a two
parameter filtration.  Let
\begin{equation*}
  \mathcal{F} = \sigma \Big(\bigcup_{\lambda \in P} \mathcal{F}_\lambda \Big).
\end{equation*}
Let $\pi$ denote the unique $\sigma$-additive measure on $(\Omega_0,
\mathcal{F})$ such that for $E_x \in \mathcal{F}_\lambda$
\begin{equation*}
	\pi(E_x) = q^{-2\sprod{\lambda}{\alpha_0}}.
\end{equation*}
All $\sigma$-fields in this paper should be extended so as to include
$\pi$-null sets.

A function $f(\omega)$ on~$\Omega_0$ is
$\mathcal{F}_\lambda$-measurable if it depends only on that part of
the apartment~$[\omega,\omega_0]$ which retracts under~$\rho$ to the
sector pointing downwards from~$\lambda$.  For $i, j \in \ZZ$ set
\begin{align*}
&\mathcal{F}_{i, \infty} = \sigma\Big(\bigcup_{j' \in \ZZ}
\mathcal{F}_{i \lambda_1 + j' \lambda_2}\Big), &
&\mathcal{F}_{\infty, j} = \sigma\Big(\bigcup_{i' \in \ZZ}
\mathcal{F}_{i'\lambda_1 + j \lambda_2}\Big).
\end{align*}
A function $f(\omega)$ on~$\Omega_0$ is
$\mathcal{F}_{i,\infty}$-measurable (respectively
$\mathcal{F}_{\infty,j}$-measurable) if it depends only on that part
of the apartment which retracts to a certain ``lower'' half-plane with
boundary parallel to~$\lambda_2$ (respectively~$\lambda_1$).

If $\mathcal{F}'$ is $\sigma$-subfield of $\mathcal{F}$, we denote by
$\EE[f | \mathcal{F}']$ the Radon--Nikodym derivative with respect to
$\mathcal{F}'$. If $\mathcal{F}''$ is another $\sigma$-subfield of
$\mathcal{F}$ we write
\begin{equation*}
	\EE[f | \mathcal{F}' | \mathcal{F}'']
	  = \EE\big[\EE[f | \mathcal{F}'  ]\big| \mathcal{F}''\big].
\end{equation*}
The $\sigma$-field generated by $\mathcal{F}' \cup \mathcal{F}''$ is
denoted by $\mathcal{F}' \vee \mathcal{F}''$. We write $f_\lambda =
\EE_\lambda f = \EE[f | \mathcal{F}_\lambda]$ for $\lambda \in P$.
If $\lambda \preceq \mu$, then it follows from
Proposition~\ref{prop:0} that
$\EE_\mu\EE_\lambda=\EE_\lambda\EE_\mu=\EE_\lambda$.

We note that the Cairoli--Walsh condition \eqref{f4} introduced in \cite{carwal} is
not satisfied, i.e.
\begin{equation*}
	\EE_{\lambda+\lambda_1} \EE_{\lambda+\lambda_2} \neq \EE_\lambda.
\end{equation*}
Instead of \eqref{f4} we have
\begin{lemma}
\label{lem:1}
For a locally integrable function $f$ on $\Omega_0$
\begin{gather}
\label{eq:1}
\EE[f_{\lambda+\lambda_1} | \mathcal{F}_{\lambda+\lambda_2} |
  \mathcal{F}_{\lambda+\lambda_1}] =q^{-1} f_{\lambda+\lambda_1}
-q^{-1} \EE[f_{\lambda+\lambda_1} |
  \mathcal{F}_{\lambda+\lambda_1-\lambda_2} \vee \mathcal{F}_\lambda]
+f_\lambda,\\
\label{eq:23}
\big(\EE_{\lambda+\lambda_2} \EE_{\lambda+\lambda_1}\big)^2 = q^{-1}
\EE_{\lambda+\lambda_2} \EE_{\lambda+\lambda_1} + (1-q^{-1})
\EE_\lambda,
\end{gather}
and likewise if we exchange~$\lambda_1$ and~$\lambda_2$.
\end{lemma}
\begin{proof}
For the proof of \eqref{eq:1} it is enough to consider $f =
\ind{E_{p_1}}$ where $p_1$ is a vertex in $\mathscr{X}$ such that
$\rho(p_1) = \lambda + \lambda_1$. Let $\mathscr{S}$ be the sector
$[p_1,\omega_0]$ and let~$x$ be the unique vertex of~$\mathscr{S}$
with $\rho(x) = \lambda$. The ball in~$\mathscr{X}$ of radius~$1$
around~$x$ has the structure of a finite projective plane.
	\begin{figure}[h]
		\includegraphics{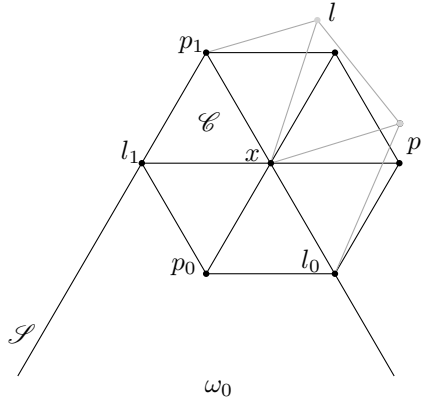}
		\caption{Residue of $x$}
		\label{fig:2}
	\end{figure}
In \textsc{Figure~\ref{fig:2}} the spot marked~$x$ is for vertices
of~$\mathscr{X}$ which retract via~$\rho$ to $\lambda$.  Recall that
$E_x$~is an atom of the $\sigma$-field~$\mathcal{F}_\lambda$. The spot
marked~$p_1$ is for vertices retracting to $\lambda+\lambda_1$; the
spot marked~$l$ is for vertices retracting to~$\lambda+\lambda_2$; the
spot marked~$l_1$ is for vertices retracting to
$\lambda+\lambda_1-\lambda_2$; etc.  In the ball of radius~$1$
around~$x$, only~$x$ itself retracts to the spot marked~$x$.  The line
type vertex known as~$l_0$ is the only vertex in the ball retracting
to its spot; $q$~line type vertices retract to the same spot as~$l_1$;
the remaining~$q^2$ line type vertices retract to the spot
marked~$l$. Likewise, $p_0$ is the unique point type vertex of the
ball retracting to its spot; $q$~point type vertices retract to the
spot marked~$p$; $q^2$~retract to the same spot as~$p_1$.  It follows
that
\begin{equation*}
	\EE[\ind{E_{p_1}} | \mathcal{F}_\lambda] =q^{-2}\ind{E_x}
    =q^{-2}\sum_{p'\not\sim l_0} \ind{E_{p'}}
    =q^{-2}\sum_{l\not\sim p_0} \ind{E_l}
\end{equation*}
and
\begin{equation*}
	\EE[\ind{E_{p_1}} |
    \mathcal{F}_{\lambda+\lambda_1-\lambda_2}\vee
    \mathcal{F}_\lambda] =q^{-1}\ind{E_x\cap E_{l_1}}
	=q^{-1}\sum_{\atop{p'\sim l_1}{p'\not\sim l_0}} \ind{E_{p'}}
\end{equation*}
where $p'$~runs through the point type vertices of the ball, $l$~runs
through the line type vertices of the ball, and $\sim$~stands for the
incidence relation. We have
\begin{equation}
\label{eq:5}
\EE[\ind{E_{p_1}} | \mathcal{F}_{\lambda+\lambda_2}] =
            q^{-1} \sum_{\atop{l \sim p_1}{l \not\sim p_0}}
                \ind{E_l}.
\end{equation}
Therefore, we obtain
\begin{equation}
	\label{eq:9}
	\begin{aligned}
\EE[\ind{E_{p_1}}|\mathcal{F}_{\lambda+\lambda_2}|\mathcal{F}_{\lambda+\lambda_1}]
  = q^{-2} \sum_{\atop{l \sim p_1}{l \not\sim p_0}}
    \sum_{\atop{p' \sim l}{p' \not\sim l_0}} 
       \ind{E_{p'}}
	&=q^{-1} \ind{E_{p_1}}
     + q^{-2}\sum_{\atop{p' \not\sim l_0}{p' \not\sim l_1}} \ind{E_{p'}} \\
 & =q^{-1}\ind{E_{p_1}} + q^{-2} \sum_{p' \not\sim l_0} \ind{E_{p'}} -q^{-2}
\sum_{\atop{p' \sim l_1}{p'\not\sim l_0}} \ind{E_{p'}},
	\end{aligned}
\end{equation}
which finishes the proof of \eqref{eq:1}. Applying one more average to
the next to the last expression of~\eqref{eq:9} we get
\begin{equation*}
\EE[\ind{E_{p_1}} | \mathcal{F}_{\lambda+\lambda_2} |
  \mathcal{F}_{\lambda+\lambda_1} | \mathcal{F}_{\lambda+\lambda_2}] =
q^{-2} \sum_{\atop{l \sim p_1}{l \not\sim p_0}} \ind{E_l} + q^{-3}
\sum_{\atop{p' \not\sim l_0}{p' \not\sim l_1}} \sum_{\atop{l \sim
    p'}{l \not\sim p_0}} \ind{E_l}.
\end{equation*}
For any line $l \not\sim p_0$ there are $q$ points $p'$ such that $p'
\sim l$ and $p' \not\sim l_0$ and among them there is exactly one
incident to $l_1$. Hence in the last sum each line $l \not\sim p_0$
appears $q-1$ times. Thus, we can write
\begin{equation*}
q^{-3} \sum_{\atop{p' \not\sim l_0}{p' \not\sim l_1}} \sum_{\atop{l
    \sim p'}{l \not\sim p_0}} = q^{-3} (q-1) \sum_{l \not \sim p_0}
\ind{E_l} = (1-q^{-1}) \EE[\ind{E_{p_1}} | \mathcal{F}_\lambda]
\end{equation*}
proving \eqref{eq:23}.
\end{proof}
The following lemma describes the composition of projections on the same level.
\begin{lemma}
\label{lem:4}
If $k, j \in \ZZ$ are such that $k \geq j \geq 0$ or $k \leq j \leq 0$ then
\begin{equation}
  \label{eq:6}
  \EE_{\lambda+k(\lambda_2-\lambda_1)} \EE_\lambda =
  \EE_{\lambda+k(\lambda_2-\lambda_1)}
  \EE_{\lambda+j(\lambda_2-\lambda_1)} \EE_\lambda.
\end{equation}
\end{lemma}
\begin{proof} We do the proof for $k\geq j\geq 0$.
  For any~$\omega\in\Omega_0$, there is a connected chain of vertices
  $\seq{x_i}{0\leq i\leq k}\subseteq[\omega,\omega_0]$ with
  $\rho(x_i)=\lambda+k(\lambda_2-\lambda_1)$.  Suppose, conversely,
  that $\seq{x_i}{0\leq i\leq k}$ is a connected chain of vertices and
  that $\rho(x_i)= \lambda+k(\lambda_2-\lambda_1)$.  Construct a
  subcomplex $\mathscr{B}\subset\mathscr{X}$ by putting together
  $\seq{[x_i,\omega_0]}{0\leq i\leq k}$, the edges between the $x_i$'s
  and the triangles pointing downwards from those edges to~$\omega_0$.
  Referring to \textsc{Figure~\ref{fig:chain}}, the extra triangle
  pointing downward from the first edge has vertices~$x_0$, $x_1$,
  and~$y_0$.  Note that
  $[x_0,\omega_0]\cap[x_1,\omega_0]=[y_0,\omega_0]$.  Proceeding one
  step at a time, one may verify that the restriction of~$\rho$
  to~$\mathscr{B}$ is an injection and that $\mathscr{B}$ and
  $\rho(\mathscr{B})$ are isomorphic complexes.
  
	\begin{figure}[h]
		\includegraphics{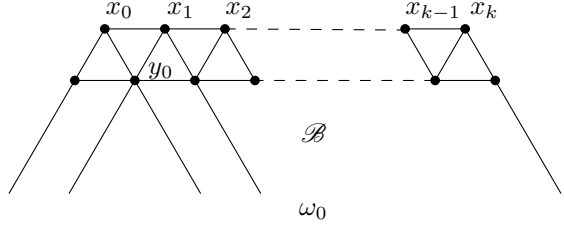}
		\caption{The complex $\mathscr{B}$}
		\label{fig:chain}
	\end{figure}

  By basic properties of affine buildings, one knows it is possible to
  extend~$\mathscr{B}$ to an apartment.  Any such apartment will
  retract bijectively to~$\Sigma$, and will be of the form
  form~$[\omega,\omega_0]$ where~$\omega$ is the equivalence class
  represented by the upward pointing sectors of the apartment.
  Moreover, using the definition of~$\pi$ one may calculate that
  \begin{equation*}
    \pi(\{\omega\in\Omega_0 : 
     \mathscr{B}\subseteq[\omega,\omega_0]\})
     =q^{-2\sprod{\lambda}{\alpha_0}-k} .
  \end{equation*}
  The important point is that the measure of the set depends only on
  the level of~$\lambda$ and the length of the chain.

  Basic properties of affine buildings imply that any apartment
  containing~$x_0$ and~$x_k$ contains the entire chain. Hence
  \begin{equation*}
    \pi(E_{x_0}\cap E_{x_k})
    =\pi(\{\omega\in\Omega_0 : 
     \mathscr{B}\subseteq[\omega,\omega_0]\})
    =q^{-2\sprod{\lambda}{\alpha_0}-k} .
  \end{equation*}
  Fix~$x_0$.  Proceeding one step at a time, one sees there are
  $q^k$~connected chains $\seq{x_i}{0\leq i\leq k}$ with
  $\rho(x_i)=\lambda+k(\lambda_2-\lambda_1)$.  Consequently
  \begin{equation*}
    \EE_{\lambda+k(\lambda_2-\lambda_1)}\ind{x_0}
      =q^{-k}\sum_{\seq{x_i}{0\leq i\leq k}}
      \ind{x_k}.
  \end{equation*}
  Likewise
  \begin{align*}
    \EE_{\lambda+k(\lambda_2-\lambda_1)}
    \EE_{\lambda+j(\lambda_2-\lambda_1)}\ind{x_0}
    & =q^{-j}\EE_{\lambda+k(\lambda_2-\lambda_1)}\sum_{\seq{x_i}{0\leq i\leq j}}
      \ind{x_j} \\
    & =q^{-j}q^{-(k-j)}\sum_{\seq{x_i}{0\leq i\leq j}}
     \sum_{\seq{x_i}{j\leq i\leq k}}
     \ind{x_k},
  \end{align*}
  which is the same thing.
\end{proof}

Consider $\EE_\lambda\EE_\mu$.  If $\lambda\preceq\mu$ then the
product is equal to~$\EE_\lambda$; similarly if~$\mu\preceq\lambda$.
If $\lambda$ and~$\mu$ are incomparable, the following lemma allows us
to reduce to the case where~$\lambda$ and~$\mu$ are on the same level.
\begin{lemma}
  \label{lem:3}
  Suppose $\lambda \in P$ and
  \begin{align*}
    \lambda' &= \lambda - i \lambda_1, &
    \mu &= \lambda' +  k(\lambda_2 - \lambda_1), &
    \tilde{\mu} &= \mu +(\lambda_2-\lambda_1)
  \end{align*}
  for $i, k \in \NN$. Then for any locally integrable function $f$ on
  $\Omega_0$
\begin{align}
  \label{eq:7.1}
  \EE[f | \mathcal{F}_\lambda | \mathcal{F}_\mu] &= \EE[f |
    \mathcal{F}_{\lambda'} | \mathcal{F}_\mu],\\
  \label{eq:7.2}
  \EE[f | \mathcal{F}_\mu | \mathcal{F}_\lambda] &= \EE[f |
    \mathcal{F}_\mu | \mathcal{F}_{\lambda'}],\\
  \label{eq:7.3}
  \EE[f | \mathcal{F}_{\lambda} | \mathcal{F}_\mu \vee
    \mathcal{F}_{\tilde{\mu}}] &= \EE[f | \mathcal{F}_{\lambda'} |
    \mathcal{F}_\mu]\\
  \label{eq:7.4}
  \EE[f | \mathcal{F}_\mu \vee \mathcal{F}_{\tilde{\mu}} |
    \mathcal{F}_\lambda] &= \EE[f | \mathcal{F}_\mu |
    \mathcal{F}_{\lambda'}]
\end{align}
and likewise if we exchange~$\lambda_1$ and~$\lambda_2$.
\end{lemma}
\begin{proof}
  We first prove~\eqref{eq:7.1} for $i=1$ and $k=1$.  Because
  $\EE[f|\mathcal{F}_{\lambda'}]
  =\EE[f|\mathcal{F}_\lambda|\mathcal{F}_{\lambda'}]$, it is sufficient
  to consider $f=\ind{E_{p_1}}$ where $\rho(p_1)=\lambda$.  Use
  \textsc{Figure~\ref{fig:2}} to fix the notation, and note that if
  $p_1$~retracts to~$\lambda$, then $x$~retracts to~$\lambda'$ and
  $p$~to $\mu$.  One calculates:
  \begin{align*}
    \EE[\ind{E_{p_1}}| \mathcal{F}_\lambda| \mathcal{F}_\mu] 
    = \EE[\ind{E_{p_1}}| \mathcal{F}_\mu] 
    = q^{-3} \sum_{\atop{p \sim l_0}{p \neq p_0}} \ind{E_p}
    &= q^{-2} \EE[\ind{E_x} | \mathcal{F}_\mu] \\
    & = \EE[\ind{E_{p_1}} | \mathcal{F}_{\lambda'} | \mathcal{F}_\mu].
  \end{align*}
  Next consider the case~$i=1$, $k>1$.  Set $\mu'=\mu+\lambda_1$, $\nu
  = \mu + \lambda_1 - \lambda_2$ and $\nu' = \nu + \lambda_1$ (see
  \textsc{Figure~\ref{fig:4}}). Since $\mathcal{F}_\mu$ is a subfield
  of $\mathcal{F}_{\mu'}$ we have
  \begin{equation*}
    \EE[f | \mathcal{F}_\lambda | \mathcal{F}_\mu] =
    \EE[f | \mathcal{F}_\lambda | \mathcal{F}_{\mu'} | \mathcal{F}_\mu].
  \end{equation*}
	\begin{figure}[h]
		\includegraphics{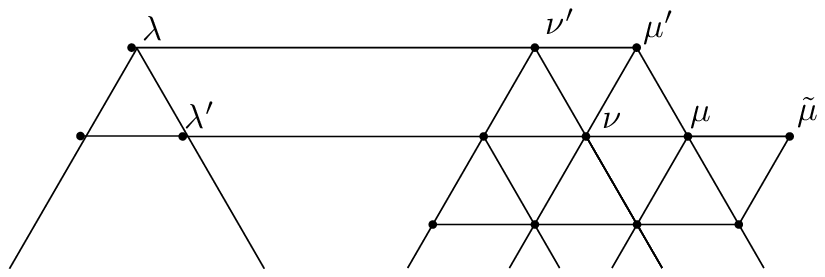}
		\caption{}
		\label{fig:4}
	\end{figure}
  Thus, applying Lemma~\ref{lem:4} we obtain
  \begin{align*}
    \EE[f |\mathcal{F}_\lambda | \mathcal{F}_\mu] 
    =\EE[f |\mathcal{F}_\lambda | 
      \mathcal{F}_{\mu'} | \mathcal{F}_\mu] 
    & = \EE[f | \mathcal{F}_\lambda | \mathcal{F}_{\nu'} |
      \mathcal{F}_{\mu'} | \mathcal{F}_\mu] \\
    & = \EE[f | \mathcal{F}_\lambda | \mathcal{F}_{\nu'} | \mathcal{F}_\mu]
    = \EE[f | \mathcal{F}_\lambda | \mathcal{F}_{\nu} | \mathcal{F}_\mu]
  \end{align*}
  where in the last step we have used the case $k = 1$.  Now apply
  induction on~$k$ and Lemma~\ref{lem:4} again to get
  \begin{equation*}
    \EE[f | \mathcal{F}_\lambda | \mathcal{F}_{\nu} | \mathcal{F}_\mu]
    = \EE[f | \mathcal{F}_{\lambda'} | \mathcal{F}_{\nu} |
      \mathcal{F}_\mu]
    = \EE[f | \mathcal{F}_{\lambda'} | \mathcal{F}_\mu]
    .
  \end{equation*}
  To extend to the case $i>1$, use induction on~$i$ and observe that
  \begin{align*}
    \EE[f | \mathcal{F}_\lambda | \mathcal{F}_\mu] 
    =\EE[f | \mathcal{F}_\lambda | \mathcal{F}_{\mu'} 
    	| \mathcal{F}_\mu]
    & =\EE[f | \mathcal{F}_{\lambda' + \lambda_1} | \mathcal{F}_{\mu'}|\mathcal{F}_\mu] \\
    & =\EE[f | \mathcal{F}_{\lambda' + \lambda_1} | \mathcal{F}_\mu]
    =\EE[f | \mathcal{F}_{\lambda'} | \mathcal{F}_\mu].
  \end{align*}
  The proof of \eqref{eq:7.3} is analogous, starting with the case
  $i=1$, $k=0$. Identity \eqref{eq:7.1} can be read as $\EE_\mu
  \EE_\lambda = \EE_\mu \EE_{\lambda'}$. The expectation operators are
  orthogonal projections with respect to the usual inner product, and
  taking adjoints gives $\EE_\lambda \EE_\mu = \EE_{\lambda'} \EE_\mu$
  which is \eqref{eq:7.2}.  To be more precise, one takes the inner
  product of either side of \eqref{eq:7.2} with some nice test
  function, applies self-adjointness, and reduces to~\eqref{eq:7.1}.
  Likewise, \eqref{eq:7.4}~follows from~\eqref{eq:7.3}.
\end{proof}

\begin{lemma}
\label{lem:5}
  Suppose $\lambda = i \lambda_1 + j \lambda_2$, $\mu = \lambda +
  k(\lambda_1 - \lambda_2)$.  Then for any locally integrable
  function $f$ on $\Omega_0$
  \begin{equation*}
    \EE[f | \mathcal{F}_\mu | \mathcal{F}_\lambda]
    =
    \begin{cases}
      \EE[f | \mathcal{F}_\mu | \mathcal{F}_{i,\infty}] &
      \text{if } k \geq 0,\\ 
      \EE[f | \mathcal{F}_\mu | \mathcal{F}_{\infty,j}] &
      \text{if } k \leq 0 
      .
    \end{cases}
  \end{equation*}
\end{lemma}
\begin{proof}
Suppose $k \geq 0$. By Lemma~\ref{lem:3} for any $j' \geq 0$ we have
\begin{equation*}
  \EE_\mu\EE_{\lambda + j' \lambda_2} =  \EE_\mu\EE_\lambda
  .
\end{equation*}
So if $g$~is $\mathcal{F}_{\lambda + j' \lambda_2}$-measurable and compactly
supported, then
\begin{equation}
  \label{eq:51}
  \begin{aligned}
  \nonumber
  \sprod{g}{\EE_{i,\infty}\EE_\mu f}
  =\sprod{\EE_\mu\EE_{i,\infty} g}{f} 
  & =\sprod{\EE_\mu g}{f} \\
  & =\sprod{\EE_\mu\EE_{\lambda+j'\lambda_2} g}{f} \\
  & =\sprod{\EE_\mu\EE_\lambda g}{f}
  =\sprod{g}{\EE_\lambda\EE_\mu f}.
  \end{aligned}
\end{equation}
The test functions~$g$ which we use are sufficient to distinguish
between one $\mathcal{F}_{i,\infty}$-measurable function and another.
Since $\EE_{i,\infty}\EE_\mu f$ and~$\EE_\lambda\EE_\mu f$ are both
$\mathcal{F}_{i,\infty}$-measurable the proof is done.
\end{proof}

\section{Littlewood-Paley theory}
\subsection{Maximal functions}
The natural maximal function $M^*$ for a locally integrable function
$f$ on $\Omega_0$ is defined by
\begin{equation*}
M^* f = \max_{\lambda \in P} \abs{f_\lambda}.
\end{equation*}
Additionally, we define two auxiliary single parameter maximal functions
\begin{align*}
L^* f &= \max_{i \in \ZZ} \EE[\,|f|\, |\, \mathcal{F}_{i, \infty}], &
R^* f &= \max_{j \in \ZZ} \EE[\,|f|\, |\, \mathcal{F}_{\infty,j}]. &
\end{align*}
\begin{lemma}
\label{lem:2}
  Let $\lambda \in P$ and~$k \in \NN$. For any non-negative locally
  integrable function~$f$ on~$\Omega_0$
  \begin{equation*}
    \big(\EE_{\lambda+k \lambda_2} \EE_{\lambda+k\lambda_1}\big)^2 f
     \geq (1-q^{-1}) \EE_\lambda f.
  \end{equation*}
\end{lemma}
\begin{proof}
  We may assume $\lambda = 0$. Let us define (see
  {\textsc{Figure~\ref{fig:5}}})
  \begin{align*}
    \mu &= k \lambda_1, 
    &\mu' &= \lambda_1 + (k-1) \lambda_2, 
    &\mu'' &= k \lambda_2,\\
    \nu &= (k-1) \lambda_1, 
    &\nu' &= \lambda_1 + (k-2) \lambda_2, 
    &\nu'' &= (k-1)\lambda_2.
  \end{align*}
  \begin{figure}[h]
    \includegraphics{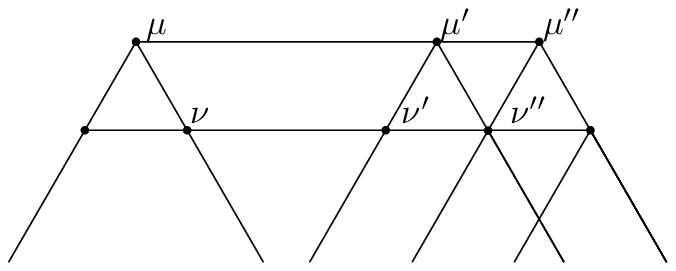}
    \caption{}
    \label{fig:5}
  \end{figure}
  We show
  \begin{equation}
    \label{eq:8}
    \EE_{\mu''} \EE_{\mu} \EE_{\mu''} \EE_{\mu} 
    - q^{-1} \EE_{\mu''} \EE_{\mu} \EE_{\mu'} \EE_{\mu}
    =
    \EE_{\nu''} \EE_{\nu} \EE_{\nu''} \EE_{\nu} 
    - q^{-1} \EE_{\nu''} \EE_{\nu} \EE_{\nu'} \EE_{\nu}.
  \end{equation}
  Let $g = \EE[f| \mathcal{F}_{\mu}]$. By two applications of
  Lemma~\ref{lem:4} we can write
  \begin{equation*}
    \EE[g | \mathcal{F}_{\mu''} | \mathcal{F}_{\mu}]
    =\EE[g | \mathcal{F}_{\mu'} | \mathcal{F}_{\mu''} | \mathcal{F}_{\mu'} 
    | \mathcal{F}_{\mu}]
  \end{equation*}
  and by Lemma~\ref{lem:1}
  \begin{equation*}
    \EE[g | \mathcal{F}_{\mu'} | \mathcal{F}_{\mu''} | \mathcal{F}_{\mu'}]
      =q^{-1} \EE[g | \mathcal{F}_{\mu'}] 
      + \EE[g | \mathcal{F}_{\nu''}]\\
      - q^{-1} \EE[g |\mathcal{F}_{\mu'} | \mathcal{F}_{\nu'} \vee
      \mathcal{F}_{\nu''}]. 
  \end{equation*}
  Hence, 
  \begin{equation*}
  \EE[g | \mathcal{F}_{\mu''} | \mathcal{F}_{\mu} | \mathcal{F}_{\mu''}]
    -q^{-1} \EE[g | \mathcal{F}_{\mu'} | \mathcal{F}_{\mu} 
    | \mathcal{F}_{\mu''}]\\
    =\EE[g | \mathcal{F}_{\nu''} | \mathcal{F}_{\mu} | \mathcal{F}_{\mu''}]
      -q^{-1} \EE[g  | \mathcal{F}_{\mu'} | \mathcal{F}_{\nu'} \vee
        \mathcal{F}_{\nu''} |\mathcal{F}_{\mu} | \mathcal{F}_{\mu''}].
  \end{equation*}
  By repeated application of Lemma~\ref{lem:3} we have
  \begin{equation*}
    \EE[g | \mathcal{F}_{\nu''} | \mathcal{F}_{\mu} |
      \mathcal{F}_{\mu''}]
    =\EE[f | \mathcal{F}_{\mu} | \mathcal{F}_{\nu''} |
      \mathcal{F}_{\mu} |  \mathcal{F}_{\mu''}]
    =\EE[f | \mathcal{F}_{\nu} | \mathcal{F}_{\nu''} | \mathcal{F}_{\nu} 
    | \mathcal{F}_{\nu''}]
  \end{equation*}
  and
  \begin{align*}
    \EE[g | \mathcal{F}_{\mu'} | \mathcal{F}_{\nu'} \vee \mathcal{F}_{\nu''} 
    | \mathcal{F}_{\mu} | \mathcal{F}_{\mu''}]
    & =\EE[f | \mathcal{F}_\mu | \mathcal{F}_{\mu'} |
       \mathcal{F}_{\nu'} \vee \mathcal{F}_{\nu''}  
      | \mathcal{F}_{\mu} | \mathcal{F}_{\mu''}] \\
    & = \EE[f | \mathcal{F}_{\nu} | \mathcal{F}_{\nu'} | \mathcal{F}_{\nu} 
    | \mathcal{F}_{\nu''}]
  \end{align*}
which finishes the proof of \eqref{eq:8}. By iteration of \eqref{eq:8} we obtain
\begin{equation*}
  \EE_{\mu''} \EE_{\mu} \EE_{\mu''} \EE_{\mu}
  - q^{-1} \EE_{\mu''} \EE_{\mu} \EE_{\mu'} \EE_{\mu}\\
  =
  \EE_{\lambda_2} \EE_{\lambda_1} \EE_{\lambda_2} \EE_{\lambda_1}
  - q^{-1} \EE_{\lambda_2} \EE_{\lambda_1} \EE_{\lambda_1} \EE_{\lambda_1}
\end{equation*}
which together with Lemma~\ref{lem:1} implies
\begin{equation*}
  \EE_{\mu''} \EE_{\mu} \EE_{\mu''} \EE_{\mu} = 
  q^{-1} \EE_{\mu''} \EE_{\mu} \EE_{\mu'}\EE_{\mu} + (1-q^{-1}) \EE_0.
	\qedhere
\end{equation*}
\end{proof}

\begin{theorem}
\label{th:1}
For each $p \in (1, \infty]$ there is $C_p > 0$ such that
\begin{gather}
  \label{eq:1.1}
  \norm{L^* f}_{L^p} \leq C_p \norm{f}_{L^p},\qquad
  \norm{R^* f}_{L^p} \leq C_p \norm{f}_{L^p},\\
  \label{eq:1.3}
  \norm{M^* f}_{L^p} \leq C_p \norm{f}_{L^p}.
\end{gather}
\end{theorem}
\begin{proof}
  Inequalities \eqref{eq:1.1} are two instances of Doob's well-known
  maximal inequality for single parameter martingales (see e.g.
  \cite{stein1}). To show \eqref{eq:1.3} consider a non-negative $f
  \in L^p(\Omega_0, \mathcal{F}_\mu)$. Fix $\lambda \in P$. Since $f
  \in L^p(\Omega_0, \mathcal{F}_{\mu'})$ for any $\mu' \succeq \mu$ we
  may assume $\mu \succeq \lambda$. Let
  \begin{equation*}
    \nu = \lambda + \sprod{\mu - \lambda}{\alpha_0} \lambda_1, 
    \qquad
    \nu'' = \lambda + \sprod{\mu - \lambda}{\alpha_0} \lambda_2.
  \end{equation*}
  By Lemma \ref{lem:2}
  \begin{equation*}
    (1- q^{-1}) \EE_\lambda f
    \leq 
    \EE_{\nu''} \EE_{\nu} \EE_{\nu''} \EE_{\nu} f.
  \end{equation*}
  If $\lambda = i \lambda_1 + j \lambda_2$, then repeated application
  of Lemma \ref{lem:5} gives
  \begin{align*}
    \EE_{\nu''} \EE_{\nu} \EE_{\nu''} \EE_{\nu} f = \EE_{\nu''}
    \EE_{\nu} \EE_{\nu''} \EE_{\nu} \EE_\mu f & = \EE[f |
      \mathcal{F}_{\infty, j} | \mathcal{F}_{i, \infty} |
      \mathcal{F}_{\infty, j} | \mathcal{F}_{i, \infty} ] \\
	& \leq L^* R^* L^*  R^* f.
  \end{align*}
  By taking the supremum over $\lambda \in P$ we get
  \begin{equation*}
    (1 - q^{-1}) M^* f 
    \leq
    L^* R^* L^* R^* f.
  \end{equation*}
  Hence, by \eqref{eq:1.1} we obtain \eqref{eq:1.3} for $f \in
  L^p(\Omega_0, \mathcal{F}_\mu)$.  Finally, a standard Fatou's lemma
  argument establishes the theorem for arbitrary $f \in
  L^p(\Omega_0)$.
\end{proof}

\subsection{Square function}
Let $f$ be a locally integrable function on $\Omega_0$. Given $i, j
\in \ZZ$ we define projections
\begin{equation*}
L_i f = \EE[f | \mathcal{F}_{i, \infty}] - \EE[f | \mathcal{F}_{i-1, \infty}],
\quad
R_j f = \EE[f | \mathcal{F}_{\infty, j}] - \EE[f | \mathcal{F}_{\infty, j-1}].
\end{equation*}
Note that $L_i$ (respectively $R_j$) is the martingale difference
operator for the filtration $\seq{\mathcal{F}_{i, \infty}}{i \in \ZZ}$
(respectively $\seq{\mathcal{F}_{\infty, j}}{j \in \ZZ}$). For
$\lambda = i \lambda_1 + j \lambda_2$ we set
\begin{equation*}
D_\lambda f = L_i R_j f, \qquad D_\lambda^\star f = R_j L_i f.
\end{equation*}
The following development is inspired by that of~Stein and~Street
in~\cite{stst}.  We start by defining the corresponding square function.
\begin{equation*}
\calS f = \Big(\sum_{\lambda \in P} \abs{D_\lambda f}^2 \Big)^{1/2}.
\end{equation*}
We will also need its dual counterpart
\begin{equation*}
\calS^\star f = \Big(\sum_{\lambda \in P} \abs{D_\lambda^\star f}^2 \Big)^{1/2}.
\end{equation*}
\begin{theorem}
  \label{th:5}
  For every $p \in (1, \infty)$ there is $C_p > 1$ such that 
  \begin{align*}
    C_p^{-1} \norm{f}_{L^p} &\leq \norm{\calS f}_{L^p} \leq C_p \norm{f}_{L^p},
    &
    C_p^{-1} \norm{f}_{L^p} &\leq \norm{\calS^\star f}_{L^p} \leq C_p \norm{f}_{L^p}.
  \end{align*}
  Moreover, on $L^2(\Omega_0)$ square functions $\calS$ and $\calS^\star$
  preserve the norm.
\end{theorem}
\begin{proof}
Since
\begin{equation*}
  S_L(f) = \Big(\sum_{i \in \ZZ} \abs{L_i f}^2\Big)^{1/2} \quad \text{and}\quad 
  S_R(f) = \Big(\sum_{j \in \ZZ} \abs{R_j f}^2\Big)^{1/2}
\end{equation*}
preserve the norm on $L^2(\Omega_0)$ we have
\begin{equation}
\label{eq:10}
	\begin{aligned}
	\int \sum_{i,j \in \ZZ} \abs{L_i R_j f}^2 \dpi
	&= \sum_{j \in \ZZ} \int \sum_{i \in \ZZ} \abs{L_i R_j f}^2 \dpi \\
	&= \sum_{j \in \ZZ} \int \abs{R_j f}^2 d\pi 
	= \int \abs{f}^2 \dpi.
	\end{aligned}
\end{equation}
Hence, $\calS$ preserves the norm.

For $p \neq 2$ we use the two parameter Khintchine inequality (see
\cite{pal}) and bounds on single parameter martingale transforms (see
\cite{burk1, stein1, weisz}). Let $\seq{\epsilon_i}{i \in \ZZ}$ and
$(\epsilon'_j : j \in \ZZ)$ be sequences of real numbers, with
absolute values bounded above by~$1$. For $N \in \NN$ we consider the
operator
\begin{equation*}
  T_N = \sum_{\abs{i}, \abs{j} \leq N} \epsilon_i \epsilon_j' D_{i\lambda_1 + j \lambda_2}
\end{equation*}
which may be written as a composition $\mathcal{L}_N \mathcal{R}_N$
where
\begin{equation*}
  \mathcal{L}_N = \sum_{\abs{i} \leq N} \epsilon_i L_i, \quad
  \mathcal{R}_N = \sum_{\abs{j} \leq N} \epsilon_j' R_j.
\end{equation*}
Since by Burkholder's inequality (see \cite{burk1, stein1}) the
operators $\mathcal{R}_N$ and 
$\mathcal{L}_N$ are bounded on
$L^p(\Omega_0)$ with bounds uniform in $N$ we have
\begin{equation*}
	\norm{T_N f}_{L^p} \lesssim \norm{f}_{L^p}.
\end{equation*}
Setting $r_k$ to be the Rademacher function, by Khintchine's
inequality we get
\begin{equation*}
  \int 
  \Big(\sum_{\abs{i}, \abs{j} \leq N}
  \abs{D_{i\lambda_1 + j \lambda_2} f}^2 \Big)^{p/2} \dpi
  \lesssim
  \int
  \int_0^1 \int_0^1
  \Big\lvert
  \sum_{\abs{i}, \abs{j} \leq N}
  r_i(s) r_j(t) D_{i \lambda_1 +j \lambda_2} f
  \Big\rvert^p {\: \rm d}s{\: \rm d} t \dpi,
\end{equation*}
which is bounded by $\norm{f}_{L^p}^p$. Finally, let~$N$ approach infinity and use the monotone convergence
theorem to get
\begin{equation*}
  \lVert \calS f \rVert_{L^p} \lesssim \lVert f \rVert_{L^p}.
\end{equation*}
For the opposite inequality, we take $f \in L^p(\Omega_0) \cap
L^2(\Omega_0)$ and $g \in L^{p'}(\Omega_0) \cap L^2(\Omega_0)$ where
$1/p'+1/p = 1$. By polarization of~\eqref{eq:10} and the
Cauchy--Schwarz and H\"older inequalities we obtain
\begin{equation*}
  \sprod{f}{g} = \int \sum_{\lambda \in P} D_\lambda f \overline{D_\lambda g} \dpi
  \leq
  \langle \calS f, \calS g \rangle \leq \norm{\calS f}_{L^p} \norm{\calS g}_{L^{p'}}
  \lesssim 
  \norm{\calS f}_{L^p}
  \norm{g}_{L^{p'}}. \qedhere
\end{equation*}
\end{proof}

Given a set $\{v_\lambda : \lambda \in P\}$ of vectors in a Banach
space, we say that $\sum_{\lambda \in P} v_\lambda$ converges
\emph{unconditionally} if, whenever we choose a bijection $\phi: \NN
\rightarrow P$,
\begin{equation*}
\sum_{n = 1}^\infty v_{\phi(n)} \qquad 
\text{exists, and is independent of~$\phi$.}
\end{equation*}
Equivalently, we may ask that for any increasing, exhaustive sequence
$\seq{F_N}{N \in \NN}$ of finite subsets of $P$, the limit
\begin{equation*}
\lim_{N \to \infty} \sum_{\lambda \in F_N} v_\lambda 
\qquad \text{exists.}
\end{equation*}
The following proposition provides a Calder\'on reproducing formula.
\begin{proposition}
  \label{prop:2}
  For each $p \in (1, \infty)$ and any $f \in L^p(\Omega_0)$,
  \begin{equation*}
    f = \sum_{\lambda \in P} D_\lambda D_\lambda^\star f
  \end{equation*}
  where the sum converges in $L^p(\Omega_0)$ unconditionally.
\end{proposition}
\begin{proof}
Fix an increasing and exhaustive sequence $\seq{F_N}{N \in \NN}$ of
finite subsets of $P$. Let
\begin{equation*}
	I_N(f) = \sum_{\lambda \in F_N} D_\lambda D_\lambda^\star f.
\end{equation*}
For $f \in L^p(\Omega_0)$ and $g \in L^{p'}(\Omega_0)$, where $1/p +
1/{p'} = 1$, we have
\begin{equation}
\begin{aligned}
  \label{eq:52}
  \abs{\sprod{I_N(f) - I_M(f)}{g}} & = 
  \Big
  \lvert
  \sum_{\lambda \in F_N \setminus F_M}
  \sprod{D_\lambda^\star f}{D_\lambda^\star g}
  \Big
  \rvert \\
  & \leq
  \Big\lVert
  \Big( \sum_{\lambda \in F_N \setminus F_M} (D_\lambda^\star f)^2 \Big)^{1/2}
  \Big\rVert_{L^p}
  \norm{\calS^\star(g)}_{L^{p'}}.
\end{aligned}
\end{equation}
In particular,
\begin{equation*}
  \abs{\sprod{I_N(f)}{g}}
  \leq
  \norm{\calS^\star(f)}_{L^p} \norm{\calS^\star(g)}_{L^{p'}},
\end{equation*}
whence $\norm{I_N(f)}_{L^p} \lesssim \norm{f}_{L^p}$ uniformly in
$N$. Consequently, it is enough to prove convergence for $f \in
L^p(\Omega_0) \cap L^2(\Omega_0)$. From \eqref{eq:52} and the bounded
convergence theorem it follows that for any positive $\epsilon$,
$\norm{I_N(f) - I_M(f)}_{L^p} \leq \epsilon$ whenever $M$ and $N$ are
large enough.  This shows that the limit exists. Finally, for~$g \in
L^{p'}(\Omega_0)\cap L^2(\Omega_0)$, the polarized version
of~\eqref{eq:10} gives
\begin{equation*}
  \lim_{N \rightarrow \infty} \sprod{I_N(f)}{g} 
  = \lim_{N \rightarrow \infty} 
  \sum_{\lambda \in F_N} \sprod{D_\lambda^\star f}{D_\lambda^\star g} 
  = \sprod{f}{g}. \qedhere
\end{equation*}
\end{proof}

\begin{theorem}
  \label{th:6}
  Let $\seq{T_\lambda}{\lambda \in P}$ be a family of operators such that for some
  $\delta > 0$ and $p_0 \in (1, 2)$
  \begin{gather}
    \pnorm{T_\lambda}{1} \lesssim 1, \label{eq:30.1} \\
    \pnorm{T_\mu T_\lambda^\star}{2} \lesssim q^{-\delta \abs{\mu-\lambda}} 
    \quad \text{and} 
    \quad \pnorm{T_\mu^\star T_\lambda}{2} \lesssim q^{-\delta \abs{\mu-\lambda}}, 
    \label{eq:30.2}\\
    \pnorm{D_\lambda T_\mu D_{\lambda'}}{2} 
    \lesssim q^{-\delta \abs{\lambda-\mu}} q^{-\delta \abs{\lambda'-\mu}},
    \label{eq:30.3}\\
    \big\lVert{\sup_{\lambda \in P} \abs{T_\lambda f_\lambda}}\big\rVert_{L^{p_0}}
    \lesssim 
	\big\lVert \sup_{\lambda \in P} \abs{f_\lambda} \big\rVert_{L^{p_0}}. 
    \label{eq:30.4}
  \end{gather}
  Then for any $p \in (p_0, 2]$ the sum $\sum_{\lambda \in P}
    T_\lambda$ converges unconditionally in the strong operator
    topology for operators on~$L^p(\Omega_0)$.
\end{theorem}
\begin{proof}
First, recall that the Cotlar--Stein Lemma (see e.g. \cite{steinone})
states that \eqref{eq:30.2}~implies the unconditional convergence of
$\sum_{\lambda \in P} T_\lambda$ in the strong operator topology on
$L^2(\Omega_0)$.  Let $\seq{F_N}{N \in \NN}$ be an arbitrary
increasing and exhaustive sequence of finite subsets of~$P$. For $N >
0$ we set
\begin{align*}
  V_N &= \sum_{\mu \in F_N} T_\mu, &
  I_N &= \sum_{\lambda \in F_N} D_\lambda D_\lambda^\star.
\end{align*}
By \eqref{eq:30.1}, \eqref{eq:30.2} and interpolation, each~$T_\mu$
is bounded on~$L^p$ for $p\in [1,2]$ and the same holds for the
finite sum~$V_N$.  We consider $f \in L^p(\Omega_0)$ for $p \in
(p_0, 2)$. By Proposition~\ref{prop:2} and Theorem \ref{th:5}, we
have
\begin{align*}
  \big\lVert V_M I_N (f) \big\rVert_{L^p} 
  & \lesssim 
  \big\lVert 
  \calS \big(V_M I_N(f) \big) 
  \big\rVert_{L^p}
  =\Big\lVert\Big(
  \sum_{\mu \in F_M} 
  \sum_{\lambda' \in F_N}
  D_\lambda T_\mu D_{\lambda'} D_{\lambda'}^\star f : \lambda \in P \Big)
  \Big\rVert_{L^p(\ell^2)}\\
  &
	=\Big\lVert \Big(
  \sum_{\gamma, \gamma' \in P}
  \ind{F_N}(\lambda + \gamma + \gamma')
  \ind{F_M}(\lambda + \gamma)
  D_\lambda T_{\lambda+\gamma} D_{\lambda+\gamma+\gamma'} 
  D_{\lambda+\gamma+\gamma'}^\star f
  : \lambda \in P \Big)
  \Big\rVert_{L^p(\ell^2)}\\
  &
  \leq
  \sum_{\gamma, \gamma' \in P}
  \big\lVert \seq{
  \ind{F_N}(\lambda+\gamma+\gamma') 
  \ind{F_M}(\lambda+\gamma)
  D_\lambda T_{\lambda+\gamma} D_{\lambda+\gamma+\gamma'} 
  D_{\lambda+\gamma+\gamma'}^\star f}{\lambda \in P}
  \big\rVert_{L^p(\ell^2)}.
\end{align*}
Finally, by change of variables we get
\begin{equation*}
  \big\lVert
  V_M I_N(f) 
  \big\rVert_{L^p} 
  \lesssim
  \sum_{\gamma,\gamma' \in P}
  \big\lVert 
  \seq{
    D_{\lambda+\gamma+\gamma'} T_{\lambda+\gamma} D_\lambda D_\lambda^\star f}
  {\lambda \in F_N}
  \big\rVert_{L^p(\ell^2)}.
\end{equation*}
Assuming there is $\delta_p > 0$ such that
\begin{equation}
  \label{eq:20}
  \norm{\seq{D_{\lambda+\gamma+\gamma'} 
  T_{\lambda+\gamma} D_\lambda f_\lambda}{\lambda \in P}}_{L^p(\ell^2)}
  \lesssim q^{-\delta_p (\abs{\gamma} + \abs{\gamma'})} 
  \norm{\seq{f_\lambda}{\lambda \in P}}_{L^p(\ell^2)}
\end{equation}
we can estimate
\begin{equation}
\begin{aligned}
  \label{eq:53}
  \big\lVert
  V_M I_N(f)
  \big\rVert_{L^p} 
  & \lesssim 
  \sum_{\gamma, \gamma' \in P} q^{-\delta_p(\abs{\gamma} + \abs{\gamma'})}
  \norm{\seq{D_\lambda^\star f}{\lambda \in F_N}}_{L^p(\ell^2)} \\
  & \lesssim
  \Big\lVert 
  \Big(\sum_{\lambda \in F_N} (D_\lambda^\star f)^2\Big)^{1/2}
  \Big\rVert_{L^p}.
\end{aligned}
\end{equation}
Theorem~\ref{th:5}, Proposition~\ref{prop:2} and~\eqref{eq:53} imply
that the~$V_M$ are uniformly bounded on~$L^p$.
	
For the proof of \eqref{eq:20}, we consider an operator $\mathcal{T}$
defined for $\vec{f} \in L^p\big(\pi, \ell^2(P)\big)$ by
\begin{equation*}
  \mathcal{T} {\vec f} = \seq{D_{\lambda+\gamma+\gamma'} 
  T_{\lambda+\gamma} D_\lambda f_\lambda}{\lambda \in P}.
\end{equation*}
Since $\pnorm{D_\lambda}{1} \lesssim 1$ and $\pnorm{T_\mu}{1} \lesssim 1$ we have
\begin{equation*}
    \big\lVert \mathcal{T} {\vec f\;}
    \big\rVert_{L^1(\ell^1)} \lesssim 
    \big\lVert \vec f\; \big\rVert_{L^1(\ell^1)}.
\end{equation*}
Also, by \eqref{eq:30.3}, we can estimate
\begin{equation*}
    \big\Vert 
    \mathcal{T} \vec f\;
    \big\rVert_{L^2(\ell^2)}^2
    = 
    \sum_{\lambda \in P} \norm{D_{\lambda+\gamma+\gamma'} T_{\lambda+\gamma} 
    D_\lambda f_\lambda}_{L^2}^2
    \lesssim
    q^{-\delta (\abs{\gamma}+\abs{\gamma'})} \sum_{\lambda \in P} \norm{f_\lambda}_{L^2}^2.
\end{equation*}
Therefore, using interpolation between $\lell{1}{1}$ and $\lell{2}{2}$
we obtain that there is $\delta' > 0$ such that
\begin{equation*}
  \big\lVert
  \mathcal{T} {\vec f}\;
  \big\rVert_{L^{p_0}(\ell^{p_0})} \lesssim q^{-\delta' (\abs{\gamma}+\abs{\gamma'})} 
  \big\lVert
  \vec f\;
  \big\rVert_{L^{p_0}(\ell^{p_0})}.
\end{equation*}
Because $\abs{D_\lambda g} \lesssim L^* R^*( \abs{g} )$, and because
Theorem~\ref{th:1} says that~$L^*$ and~$R^*$ are bounded on~$L^{p_0}$,
we know that $\seq{D_\lambda}{\lambda \in P}$ is bounded
on~$L^{p_0}(\pi, \ell^\infty(P))$.  Of course the same holds for
$\seq{D_{\lambda+\gamma+\gamma'}}{\lambda \in P}$.
Hence, by \eqref{eq:30.4} we get
\begin{equation*}
  \big\lVert 
  \mathcal{T} {\vec f}\;
  \big\rVert_{L^{p_0}(\ell^\infty)} 
  \lesssim \big\lVert \vec f\;\big\rVert_{L^{p_0}(\ell^\infty)}.
\end{equation*}
Next, interpolating between $\lell{p_0}{p_0}$ and
$\lell{p_0}{\infty}$ gives a~$\delta'' > 0$ such that
\begin{equation*}
  \big\lVert
  \mathcal{T} {\vec f}\;
  \big\rVert_{L^{p_0}(\ell^2)}
  \lesssim 
  q^{-\delta''(\abs{\gamma} + \abs{\gamma'})}
  \big\lVert \vec f\;\big\rVert_{L^{p_0}(\ell^2)}.
\end{equation*}
Finally, interpolating between $\lell{p_0}{2}$ and $\lell{2}{2}$ we obtain
\eqref{eq:20}.

To finish the proof, we are going to show that $\seq{V_N f}{N \in
  \NN}$ is a Cauchy sequence in $L^p(\Omega_0)$. Let us consider $g
\in L^p(\Omega_0) \cap L^2(\Omega_0)$. Setting
\begin{equation*}
  a = \frac{2(p-p_0)}{4-p-p_0},\quad \text{and} \quad \tilde{p} =
  \frac{p + p_0}{2} 
\end{equation*}
and using the log-convexity of the $L^q$-norms we get 
\begin{equation*}
\big\lVert V_M g - V_N g \big\rVert_{L^p}^p
\leq
\big\lVert V_M g - V_N g \big\rVert_{L^2}^a
\big\lVert V_M g - V_N g \big\rVert_{L^{\tilde{p}}}^{p-a}.
\end{equation*}
Since $\seq{V_N g}{N \in \NN}$ converges in $L^2(\Omega_0)$ and is
uniformly bounded on $L^{\tilde{p}}(\Omega_0)$ it is a Cauchy sequence
in $L^p(\Omega_0)$. For an arbitrary $f \in L^p(\Omega_0)$ use the
density of $g$'s as above.  We have
\begin{equation*}
  \norm{V_M f - V_N f}_{L^p} \lesssim  \norm{f - g}_{L^p} + \norm{V_N g - V_M g}_{L^p}.
\end{equation*}
Thus $\seq{V_N f}{N\in\NN}$ also converges, and this finishes the
proof of the theorem.
\end{proof}

\section{Double Differences}
The martingale transforms are expressed in terms of double differences
defined for a martingale $f = \seq{f_\lambda}{\lambda \in P}$ as
\begin{equation*}
  d_\lambda f= f_{\lambda} - f_{\lambda-\lambda_1} - f_{\lambda - \lambda_2} 
    + f_{\lambda-\lambda_1 - \lambda_2}.
\end{equation*}
\subsection{Martingale transforms}
The following proposition is our key tool.
\begin{proposition}
	\label{prop:1}
	Let $f \in L^2(\Omega_0)$ and $\lambda \in P$. If $f_{\lambda - j \lambda_1} = 0$
	for $j \in \NN$ then for each $k \geq j$
	$$
	\norm{\EE[f_\lambda | \mathcal{F}_{\lambda-k(\lambda_1-\lambda_2)}]}_{L^2} 
	\leq 2 q^{-(k-j+1)/2}\norm{f_\lambda}_{L^2}.
	$$
	Analogously, for $\lambda_1$ and $\lambda_2$ exchanged.
\end{proposition}
\begin{proof}
	Suppose $j = 1$. We are going to show that if $f_{\lambda - \lambda_1} = 0$ then for all $k \geq 1$
	\begin{equation}
		\label{eq:2}
		\norm{\EE[f_\lambda | \mathcal{F}_{\lambda - k(\lambda_1 - \lambda_2)} ]}_{L^2}
		\leq
		q^{-k/2}
		\norm{f_\lambda}_{L^2}.
	\end{equation}
	Indeed, if $k = 1$ then by \eqref{eq:1} of Lemma~\ref{lem:1}
	\begin{align*}
	\big\lVert \EE[f_\lambda|\mathcal{F}_{\lambda-\lambda_1+\lambda_2}] \big\rVert_{L^2}^2
	& =\sprod{\EE[f_\lambda | \mathcal{F}_{\lambda-\lambda_1+\lambda_2} 
	| \mathcal{F}_\lambda]}{f_\lambda} \\
	& = q^{-1} \norm{f_\lambda}^2_{L^2} 
	- q^{-1} \big\lVert \EE[f_\lambda | \mathcal{F}_{\lambda-\lambda_1} \vee 
	\mathcal{F}_{\lambda-\lambda_2}] \big\rVert_{L^2}^2.
	\end{align*}
	If $k > 1$, we use Lemma~\ref{lem:4} to write
	$$
	\EE[f_\lambda | \mathcal{F}_{\lambda-k(\lambda_1-\lambda_2)}] 
	= \EE[f_\lambda |\mathcal{F}_{\lambda-(\lambda_1-\lambda_2)} 
	| \mathcal{F}_{\lambda-k(\lambda_1-\lambda_2)}].
	$$
	Since, by Lemma~\ref{lem:3},
	$$
	\EE[f_\lambda | \mathcal{F}_{\lambda-(\lambda_1-\lambda_2)} 
	| \mathcal{F}_{\lambda-\lambda_1 - (\lambda_1 - \lambda_2)}] 
	= \EE[f_\lambda | \mathcal{F}_{\lambda-\lambda_1} | 
	\mathcal{F}_{\lambda-\lambda_1- (\lambda_1 - \lambda_2)}] = 0
	$$
	we can use induction to obtain
	\begin{align*}
	\norm{\EE[f_\lambda | \mathcal{F}_{\lambda-(\lambda_1-\lambda_2)} 
	| \mathcal{F}_{\lambda-k(\lambda_1-\lambda_2)}]}_{L^2} 
	& \leq q^{-(k-1)/2} 
	\big\lVert \EE[f_\lambda | \mathcal{F}_{\lambda-(\lambda_1-\lambda_2)}] \big\rVert_{L^2}  \\
	& \leq q^{-k/2} \norm{f_\lambda}_{L^2}.
	\end{align*}
	Let us consider $j > 1$. For each $i = 0, 1, \ldots, j-1$ we set
	$$
	g_i = f_{\lambda - i \lambda_1} - f_{\lambda - (i+1) \lambda_1}.
	$$
	By Lemma~\ref{lem:3} and \eqref{eq:2}
	\begin{align*}
	\norm{\EE[g_i | \mathcal{F}_{\lambda-k(\lambda_1 - \lambda_2)}]}_{L^2}
	=\norm{\EE[g_i | \mathcal{F}_{\lambda - k(\lambda_1- \lambda_2) - i \lambda_2}]}_{L^2}
	& \leq q^{-(k-i)/2} \norm{g_i}_{L^2} \\
	& \leq q^{-(k-i)/2} \norm{f_\lambda}_{L^2}.
	\end{align*}
	Hence,
	\begin{align*}
	\big\lVert
	\EE[f_\lambda | \mathcal{F}_{\lambda - k(\lambda_1 - \lambda_2)}]
	\big\rVert_{L^2} 
	& \leq \sum_{i=0}^{j-1} 
	\big\lVert 
	\EE[g_i | \mathcal{F}_{n- k(\lambda_1-  \lambda_2)}]
	\big\rVert_{L^2} \\
	& \leq \sum_{i = 0}^{j-1} q^{-(k-i)/2} \norm{f_\lambda}_{L^2}
	\end{align*}
	which finishes the proof since
	\[
		\sum_{i = 0}^{j-1} q^{i/2} \leq 2 q^{(j-1)/2}. \qedhere
	\]
\end{proof}
We have the following
\begin{proposition}
	\label{prop:3}
	For any $\lambda, \lambda', \mu \in P$ and $m \geq 1$
	\begin{align*}
		\pnorm{D_\lambda d_\mu^m D_{\lambda'}}{2} & \lesssim q^{-\abs{\mu-\lambda}/4} 
		q^{-\abs{\mu-\lambda'}/4},\\
		\pnorm{d_\lambda^m d_\mu^m}{2} & \lesssim q^{-\abs{\lambda-\mu}/2}.
	\end{align*}
\end{proposition}
\begin{proof}
	We observe that for $f \in L^2(\Omega_0)$, $d_\mu f \in L^2(\pi, \mathcal{F}_\mu)$ and 
	\begin{equation}
		\label{eq:22}
		\EE[d_\mu f| \mathcal{F}_{\nu}]=0
	\end{equation}
	whenever $\sprod{\nu}{\alpha_0} \leq \sprod{\mu}{\alpha_0} - 2$. For the proof it is enough to
	analyze the case $\nu = \mu - 2 \lambda_2$. By Lemma~\ref{lem:3}, we can write
	$$
	\EE[f_{\mu-\lambda_1} | \mathcal{F}_{\mu - 2 \lambda_2}] 
	= \EE[f_{\mu-\lambda_1} | \mathcal{F}_{\mu-\lambda_1-\lambda_2} 
	|\mathcal{F}_{\mu-2\lambda_2}] = \EE[f_{\mu-\lambda_1-\lambda_2} 
	| \mathcal{F}_{\mu-2\lambda_2}].
	$$
	Suppose $\lambda = i \lambda_1 + j \lambda_2$. Let us consider $R_j d_\mu$. 
	If $j \geq \sprod{\mu}{\alpha_2} + 1$ then $R_j d_\mu f = 0$. For $j \leq \sprod{\mu}{\alpha_2} - 2$,
	in view of \eqref{eq:22} we can use Proposition~\ref{prop:1} to estimate
	\begin{equation}
		\label{eq:15}
		\norm{R_j d_\mu f}_{L^2} \lesssim q^{-\sprod{\mu - \lambda}{\alpha_2}/2} \norm{d_\mu f}_{L^2}.
	\end{equation}
	Next, if $\sprod{\lambda}{\alpha_0} \geq \sprod{\mu}{\alpha_0} + 2$ then $D_\lambda d_\mu f = 0$, because
	$d_\mu f$ is $\calF_\mu$-measurable. For $\sprod{\lambda}{\alpha_0} \leq \sprod{\mu}{\alpha_0} - 4$ and
	$\sprod{\lambda}{\alpha_2} \leq \sprod{\mu}{\alpha_2}$, by Lemma \ref{lem:5} we can write
	$D_\lambda d_\mu f = L_i g$ where 
	$$
	g = \EE[R_j d_\mu f | \mathcal{F}_{\nu}]
	$$
	and $\nu = (\sprod{\mu}{\alpha_0} - j)\lambda_1 + j \lambda_2$. By Lemma \ref{lem:5}, we have
	\[
		R_j d_\mu f = \EE[d_\mu f | \calF_\nu] - \EE[d_\mu f | \calF_{\nu+\lambda_1-\lambda_2}].
	\]
	We notice that by Lemma \ref{lem:3} and \eqref{eq:22}
	$$
	\EE[d_\mu f | \mathcal{F}_{\nu} | \mathcal{F}_{\nu - 2 \lambda_1}] = 
	\EE[d_\mu f | \mathcal{F}_{\mu - 2\lambda_2} | \mathcal{F}_{\nu - 2 \lambda_1}] = 0.
	$$
	Similarly, one can show 
	$$
	\EE[d_\mu f | \mathcal{F}_{\nu + \lambda_1 - \lambda_2} | \mathcal{F}_{\nu - 2\lambda_1}] = 0.
	$$
	Therefore, $\EE[g | \mathcal{F}_{\nu - 2\lambda_1}] = 0$. Now, by Proposition~\ref{prop:1},
	we obtain
	\begin{equation}
		\label{eq:25}
		\norm{L_i g}_{L^2} \lesssim q^{-\sprod{\nu-\lambda}{\alpha_0}/2} \norm{R_j d_\mu f}_{L^2}.
	\end{equation}
	Combining \eqref{eq:25} with \eqref{eq:15} we get
	\begin{equation}
		\label{eq:16}
		\norm{D_\lambda d_\mu f}_{L^2} 
		\lesssim 
		q^{-\sprod{\mu - \lambda}{\alpha_0}/2}
		q^{-\sprod{\mu - \lambda}{\alpha_2}/2} 
		\norm{d_\mu f}_{L^2}
	\end{equation}
	since $\sprod{\nu}{\alpha_0} = \sprod{\mu}{\alpha_0}$. By analogous reasoning one can show the corresponding
	norm estimates for $D_{\lambda'}^\star d_\mu$. Hence, taking adjoint
	\begin{equation}
		\label{eq:17}
		\norm{d_\mu D_{\lambda'} f}_{L^2} 
		\lesssim 
		q^{-\sprod{\mu - \lambda'}{\alpha_0}/2} 
		q^{-\sprod{\mu - \lambda'}{\alpha_2}/2} 
		\norm{f}_{L^2}.
	\end{equation}
	Finally, \eqref{eq:16} and \eqref{eq:17} allow us to conclude the proof of the first
	inequality.

	For the second, we may assume $0 \leq \sprod{\mu - \lambda}{\alpha_0} \leq 1$. Suppose 
	$\sprod{\mu - \lambda}{\alpha_0} = 0$ and 
	$\sprod{\mu - \lambda}{\alpha_2} \geq 2$. Since 
	$d_\mu f \in L^2(\pi, \mathcal{F}_\mu)$, by \eqref{eq:22} and Proposition~\ref{prop:1}
	$$
	\norm{\EE[d_\mu f| \mathcal{F}_\lambda]}_{L^2} 
	\lesssim 
	q^{-\sprod{\mu - \lambda}{\alpha_2}/2} 
	\norm{d_\mu f}_{L^2}.
	$$
	Similarly, we deal with the case $\sprod{\mu - \lambda}{\alpha_0} = 1$. We can assume 
	$\sprod{\mu - \lambda}{\alpha_2} \geq 1$. By Lemma~\ref{lem:3}, we have 
	$$
	\EE[d_\mu f | \mathcal{F}_\lambda] 
	= \EE[d_\mu f | \mathcal{F}_{\mu-\lambda_2} | \mathcal{F}_\lambda]
	= \EE[f_{\lambda-\lambda_1-\lambda_2} - f_{\lambda-\lambda_1} | \mathcal{F}_\lambda].
	$$
	Hence, by Proposition~\ref{prop:1}
	\[
	\norm{\EE[d_\mu f| \mathcal{F}_\lambda]}_{L^2} 
	\lesssim 
	q^{-\sprod{\mu - \lambda}{\alpha_2}/2} \norm{f}_{L^2}. 
	\qedhere
	\]
\end{proof}

Let $\seq{a_\lambda}{\lambda \in P}$ be an uniformly bounded \emph{predictable} family of
functions, i.e. each function $a_\lambda$ is measurable with respect to $\mathcal{F}_{\lambda-\lambda_1-\lambda_2}$ and
$$
\sup_{\omega \in \Omega_0} \abs{a_\lambda(\omega)} \leq M.
$$
Predictability is the condition needed to ensure that 
$d_\lambda \big(a_\lambda f\big) = a_\lambda d_\lambda f$. By Theorem~\ref{th:6},
Theorem~\ref{th:1}, Proposition~\ref{prop:3} and duality when $p > 2$, we get
\begin{theorem}
	\label{th:4}
	For each $p \in (1, \infty)$ and $m \in \NN$ the series
	$$
	\sum_{\lambda \in P} a_\lambda d_\lambda^m
	$$
	converges unconditionally in the strong operator topology for the operators on $L^p(\Omega_0)$, and
	defines the operator with norm bounded by a constant multiply of
	$$
	\sup_{\lambda \in P} \sup_{\omega \in \Omega_0} \abs{a_\lambda(\omega)}.
	$$
\end{theorem}

\subsection{Martingale square function}
For a martingale $f = \seq{f_\lambda}{\lambda \in P}$ there is the natural square function
defined by
$$
S f = \Big(\sum_{\lambda \in P} (d_\lambda f)^2 \Big)^{1/2}.
$$
Although $S$ does not preserve $L^2$ norm we have
\begin{theorem}
	\label{th:2}
	For every $p \in (1, \infty)$ there is $C_p > 0$ such that
	$$
	C_p^{-1} \norm{f}_{L^p} \leq \norm{S f}_{L^p} \leq C_p \norm{f}_{L^p}.
	$$
\end{theorem}
\begin{proof}
	We start from proving the identity
	\begin{equation}
		\label{eq:48}
		d_\lambda^4 - d_\lambda^3 - q^{-1} d_\lambda^2 + q^{-1} d_\lambda =0.
	\end{equation}
	Let us notice that
	\begin{align*}
		d_\lambda \EE_{\lambda} &= d_\lambda,
		&d_\lambda \EE_{\lambda-\lambda_1-\lambda_2} &= 0, \\
		d_\lambda \EE_{\lambda-\lambda_2} &= - \EE_{\lambda-\lambda_1} \EE_{\lambda-\lambda_2} 
		+ \EE_{\lambda-\lambda_1-\lambda_2},
		&d_\lambda \EE_{\lambda-\lambda_1} &= -\EE_{\lambda-\lambda_2} \EE_{\lambda-\lambda_1} 
		+\EE_{\lambda-\lambda_1 - \lambda_2}.
	\end{align*}
	Therefore, consecutively we have
	\begin{align}
		\label{eq:46}
		d_\lambda^2 &= d_\lambda + \EE_{\lambda-\lambda_1} \EE_{\lambda-\lambda_2} 
		+ \EE_{\lambda - \lambda_1} \EE_{\lambda - \lambda_1} - 2 \EE_{\lambda-\lambda_1-\lambda_2},\\
		\nonumber
		d_\lambda^3 &= d_\lambda^2 - \EE_{\lambda-\lambda_1} \EE_{\lambda-\lambda_2} 
		\EE_{\lambda - \lambda_1} - \EE_{\lambda-\lambda_2}\EE_{\lambda-\lambda_1}
		\EE_{\lambda- \lambda_2}
		+2\EE_{\lambda-\lambda_1-\lambda_2},\\
		\nonumber
		d_\lambda^4 &= d_\lambda^3 
		+ (\EE_{\lambda-\lambda_1} \EE_{\lambda - \lambda_2})^2 
		+ (\EE_{\lambda-\lambda_2} \EE_{\lambda - \lambda_1})^2 - 2\EE_{\lambda-\lambda_1-\lambda_2}.
	\end{align}
	Hence, by Lemma \ref{lem:1},
	\[
		d_\lambda^4 = d_\lambda^3 + q^{-1} \EE_{\lambda-\lambda_1} \EE_{\lambda - \lambda_2}
		+ q^{-1} \EE_{\lambda-\lambda_2} \EE_{\lambda-\lambda_1} - 2 q^{-1} \EE_{\lambda-\lambda_1-\lambda_2}
	\]
	which together with \eqref{eq:46} implies \eqref{eq:48}.

	Next, we consider an operator $\mathcal{T}$ defined for a function $f \in L^p(\Omega_0)$ by
	$$
	\mathcal{T} f = \seq{d_\lambda f}{\lambda \in P}.
	$$
	We also need an operator $\widetilde{\mathcal{T}}$ acting on $g \in L^{p'}(\Omega_0)$ as
	$$
	\widetilde{\mathcal{T}}g 
	= \seq{-q d_\lambda^3 g + q d_\lambda^2 g + d_\lambda g}{\lambda \in P}.
	$$
	We observe that by two parameter Khinchine's inequality and Theorem~\ref{th:4} we have
	$$
	\big\lVert
	\mathcal{T} f 
	\big\rVert_{L^p(\ell^2)} \lesssim \norm{f}_{L^p}, \quad \text{and}\quad
	\big\lVert
	\widetilde{\mathcal{T}} g
	\big\rVert_{L^{p'}(\ell^2)} \lesssim \norm{g}_{L^{p'}}.
	$$
	The dual operator
	$\mathcal{T}^\star: L^{p'}\big(\pi, \ell^2(\ZZ^2)\big) \rightarrow L^{p'}(\Omega_0)$
	is given by
	$$
	\mathcal{T}^\star \vec{g} = \sum_{\lambda \in P} d_\lambda g_\lambda.
	$$
	Since $\widetilde{\mathcal{T}} g \in L^{p'}\big(\pi, \ell^2(\ZZ^2)\big)$,
	by \eqref{eq:48} and Theorem~\ref{th:4},
	$$
	\mathcal{T}^\star \widetilde{\mathcal{T}} g = \sum_{\lambda \in P} d_\lambda g = g
	$$
	Therefore, by Cauchy--Schwarz and H\"older inequalities
	$$
	\sprod{f}{g} = \sprod{f}{\mathcal{T}^\star \widetilde{\mathcal{T}} g}
	\leq
	\big\lVert
	\mathcal{T} f
	\big\rVert_{L^p(\ell^2)}
	\big\lVert
	\widetilde{\mathcal{T}} g
	\big\rVert_{L^{p'}(\ell^2)}
	\lesssim
	\big\lVert
	\mathcal{T} f
	\big\rVert_{L^p(\ell^2)}
	\norm{g}_{L^{p'}}
	$$
	and since $\norm{\mathcal{T}f}_{L^p(\ell^2)} = \norm{S f}_{L^p}$ the proof is finished.
\end{proof}

Finally, the method of the proof of Theorem \ref{th:6}, together with Theorem \ref{th:4} and Theorem \ref{th:2}
shows the following
\begin{theorem}
	Let $\seq{T_\lambda}{\lambda \in P}$ be a family of operators such that for some
	$\delta > 0$ and $p_0 \in (1, 2)$
	\begin{gather*}
		\pnorm{T_\lambda}{1} \lesssim 1, \\
		\pnorm{T_\mu T_\lambda^\star}{2} \lesssim q^{-\delta \abs{\mu-\lambda}} 
		\quad \text{and} 
		\quad \pnorm{T_\mu^\star T_\lambda}{2} \lesssim q^{-\delta \abs{\mu-\lambda}},\\
		\pnorm{d_\lambda T_\mu d_{\lambda'}}{2} 
		\lesssim q^{-\delta \abs{\lambda-\mu}} q^{-\delta \abs{\lambda'-\mu}},\\
		\big\lVert{\sup_{\lambda \in P} \abs{T_\lambda f_\lambda}}\big\rVert_{L^{p_0}} 
		\lesssim 
		\big\lVert \sup_\lambda \abs{f_\lambda} \big\rVert_{L^{p_0}}.
	\end{gather*}
	Then for any $p \in (p_0, 2]$ the sum $\sum_{\lambda \in P} T_\lambda$ converges
	unconditionally in the strong operator topology for the operators on $L^p(\Omega_0)$.
\end{theorem}

\appendix
\section{About $\Omega_0$ and Heisenberg group}
\label{ap:1}
	In some cases $\Omega_0$ can be identified with a Heisenberg group over a nonarchimedean local field. Let
	us recall, that $F$ is a \emph{nonarchimedean local field} if it is a topological field
	\footnote{A \emph{topological field} is an algebraic field with a topology making addition, multiplication and
	multiplicative inverse a continuous mappings.} that is locally compact, second
	countable, non-discrete and totally disconnected. Since $F$ together with the additive structure is a locally compact
	topological group it has a Haar measure $\mu$ that is unique up to multiplicative constant. Observe that for each
	$x \in F$, the measure $\mu_x(B) =  \mu(x B)$ is also a Haar measure. We set
	\[
		\abs{x} = \frac{\mu_x(B)}{\mu(B)},
	\]
	where $B$ is any measurable set with finite and positive measure.
	By $\calO = \{x \in F : \abs{x} \leq 1\}$, we denote the ring of integers in $F$. We fix $\pi \in \mathfrak{p} -
	\mathfrak{p}^2$, where
	\[
		\mathfrak{p} = \big\{x \in F : \abs{x} < 1\big\}.
	\]
	We are going to sketch the construction of a building associated to $\GL(3, F)$. For more details we refer
	to \cite{st-b}. A lattice is a subset $L \subset F^3$ of the form
	\[
		L = \calO v_1 + \calO v_2 + \calO v_3,
	\]
	where $\{v_1, v_2, v_3\}$ is a basis of $F^3$.
	We say that two lattices $L_1$ and $L_2$ are equivalent if and only if $L_1 = a L_2$ for some nonzero
	$a \in F$. Then $\mathscr{X}$, the building of $\GL(3, F)$, is the set of equivalence
	classes of lattices in $F^3$. For $x, y \in \mathscr{X}$ there are a basis $\{v_1, v_2, v_3\}$
	of $F^3$ and integers $j_1 \leq j_2 \leq j_3$ such that (see \cite[Proposition 3.1]{st-b})
	\[
		x = \calO v_1 + \calO v_2 + \calO v_3, \quad\text{and}\quad
		y = \pi^{j_1} \calO v_1 + \pi^{j_2} \calO v_2 + \pi^{j_3} \calO v_3.
	\]
	We say that $x$ and $y$ are joined by an edge if and only if $0 = j_1 \leq j_2 \leq j_3 = 1$. The subset
	\[
		\mathscr{A} =
		\big\{\pi^{j_1} \calO v_1 + \pi^{j_2} \calO v_2 + \pi^{j_3} \calO v_3 : j_1, j_2, j_3 \in \ZZ \big\}
	\]
	is called an apartment. A sector in $\mathscr{A}$ is a subset of the form
	\[
		\calS = \big\{x + \pi^{j_1} \calO v_1 + \pi^{j_2} \calO v_2 + \pi^{j_3} \calO v_3 : 
		j_{\sigma(1)}\leq j_{\sigma(2)} \leq j_{\sigma(3)}, j_1, j_2, j_3 \in \ZZ \big\},
	\]
	where $\sigma$ is a permutation of $\{1, 2, 3\}$ and $x \in \mathscr{A}$. Thus, a subsector of $\calS$ is
    \[
        \big\{x + \pi^{k_1+j_1} \calO v_1 + \pi^{k_2+j_2} \calO v_2 + \pi^{k_3+ j_3} \calO v_3 : j_{\sigma(1)} 
		\leq j_{\sigma(2)} \leq j_{\sigma(3)},
        j_1, j_2, j_3 \in \ZZ\},
    \]
	for some $0 \leq k_{\sigma(1)} \leq k_{\sigma(2)} \leq k_{\sigma(3)}$. Finally, two sectors
	\[
		\calS = \big\{x + \pi^{j_1} \calO v_1 + \pi^{j_2} \calO v_2 + \pi^{j_3} \calO v_3 :
        j_{\sigma(1)}\leq j_{\sigma(2)} \leq j_{\sigma(3)}, j_1, j_2, j_3 \in \ZZ \big\},
	\]
	and
	\[
		\calS' = \big\{x' + \pi^{j_1} \calO v_1 + \pi^{j_2} \calO v_2 + \pi^{j_3} \calO v_3 :
        j_{\sigma'(1)}\leq j_{\sigma'(2)} \leq j_{\sigma'(3)}, j_1, j_2, j_3 \in \ZZ \big\},
	\]
	are opposite if $\sigma' \circ \sigma^{-1} = (3\, 2\, 1)$.

	A sector in $\mathscr{X}$ is a sector in one of its apartments. Two sectors in $\mathscr{X}$ are equivalent
	if and only if its intersection contains a sector. By $\Omega$ we denote the equivalence classes of sectors in 
	$\mathscr{X}$. Let $\omega_0$ and $\omega_0'$ be the equivalence class of
	\[
		\mathscr{S}_0 = \big\{\pi^{j_1} \calO e_1 + \pi^{j_2} \calO e_2 + \pi^{j_3} \calO e_3 :
        j_1 \leq j_2 \leq j_3, j_1, j_2, j_3 \in \ZZ \big\},
	\]
	and
	\[
		\mathscr{S}_0' = \big\{\pi^{j_1} \calO e_1 + \pi^{j_2} \calO e_2 + \pi^{j_3} \calO e_3 :
        j_1 \geq j_2 \geq j_3, j_1, j_2, j_3 \in \ZZ \big\},
	\]
	respectively. Two sectors $\mathscr{S}$ and $\mathscr{S}'$ are opposite in $\mathscr{X}$ if there are subsectors of
	$\mathscr{S}$ and $\mathscr{S}'$ opposite in a common apartment. By $\Omega_0$ we denote the equivalence classes
	of sectors opposite to $\mathscr{S}_0$. 

	Suppose that $\omega' \in \Omega_0$. Let $\{v_1, v_2, v_3\}$ be a basis of $F^3$, and $k_1 \leq k_2 \leq k_3$ and
	$k_1' \geq k_2' \geq k_3'$ be integers such that
	\begin{equation}
		\label{eq:a1}
		\big\{\pi^{j_1 + k_1} \calO v_1 + \pi^{j_2 + k_2} \calO v_2 + \pi^{j_3 + k_3} \calO v_3 : 
		j_1 \leq j_2 \leq j_3, j_1, j_2, j_3 \in \ZZ\big\},
	\end{equation}
	and
	\begin{equation}
		\label{eq:a2}
		\big\{\pi^{j_1 + k_1'} \calO v_1 + \pi^{j_2 + k_2'} \calO v_2 + \pi^{j_3 + k_3'} \calO v_3 :
        j_1 \geq j_2 \geq j_3, j_1, j_2, j_3 \in \ZZ\big\},
	\end{equation}
	belong to $\omega_0$ and $\omega'$, respectively. Since the sector \eqref{eq:a1} belongs to $\omega_0$, we have
	\[
		v_1 = b_{11} e_1, \quad
		v_2 = b_{21} e_1 + b_{22} e_2, \quad
		v_3 = b_{31} e_1 + b_{32} e_2 + b_{33} e_3,
	\]
	for some $b_{ij} \in F$ such that $b_{11}, b_{22}, b_{33} \neq 0$. Hence, the matrix
	\[
		g = 
		\begin{pmatrix}
			b_{11} & b_{21} & b_{31} \\
			0 & b_{22} & b_{32} \\
			0 & 0 & b_{33}
		\end{pmatrix},
	\]
	satisfies $g e_j = v_j$. In particular, $g\omega'_0 = \omega'$. Therefore, the group of upper triangular matrices acts
	transitively on $\Omega_0$. Observe also that the stabilizer of $\omega_0'$ in $\GL(3, F)$ is
	the group of lower triangular matrices. Thus the group
	\[
		\left\{
		\begin{pmatrix}
		1 & x & z \\
		0 & 1 & y \\
		0 & 0 & 1
		\end{pmatrix} :
		x, y, z \in F
		\right\}
	\]
	acts simply transitively on $\Omega_0$.

\end{document}